\documentclass[amstex,12pt,reqno]{amsart}
\usepackage{amsmath, amssymb, amsthm}
\usepackage{mathrsfs}

\usepackage{mathtools}
\mathtoolsset{showonlyrefs=true}

\theoremstyle{plain}
\newtheorem{theorem}{Theorem}[section]
\newtheorem{proposition}[theorem]{Proposition}
\newtheorem{lemma}[theorem]{Lemma}

\newtheorem{corollary}[theorem]{Corollary}

\theoremstyle{definition}
\newtheorem{remark}{Remark}[section]
\newtheorem{definition}{Definition}[section]

\title[Integrable Teich\-m\"ul\-ler spaces and Weil--Petersson curves]
{Integrable Teich\-m\"ul\-ler spaces \\for analysis on Weil--Petersson curves}
\author[K. Matsuzaki]{Katsuhiko Matsuzaki}
\address{Department of Mathematics, School of Education, Waseda University \endgraf
Shinjuku, Tokyo 169-8050, Japan}
\email{matsuzak@waseda.jp}



\thanks{Research supported by supported by 
Japan Society for the Promotion of Science (KAKENHI 23K25775 and 23K17656).}

\begin{document}

\maketitle

\begin{abstract}
The integrable Teich\-m\"ul\-ler space $T_p$ for $p \geq 1$ is defined by the $p$-integrability of Beltrami coefficients. 
We characterize a quasisymmetric homeomorphism $h$ in $T_p$ by the condition that $\log h'$ belongs to the real $p$-Besov space, 
with a certain modification applied in the case $p=1$. This is done as part of the arguments 
for establishing a biholomorphic correspondence $\Lambda$ from the product of $T_p$ 
for simultaneous uniformization of $p$-Weil--Petersson curves into the $p$-Besov space. 
In particular, this proves the real-analytic equivalence between $T_p$ and the real $p$-Besov space. 
Moreover, the Cauchy transform of Besov functions on Weil--Petersson curves 
can be expressed by the derivative of this holomorphic map $\Lambda$, and from this, the Calder\'on theorem in this setting is straightforward. It also follows that the Cauchy transforms on $p$-Weil--Petersson curves 
holomorphically depend on their embeddings as they vary in the Bers coordinates.
\end{abstract}

\section{Introduction}\label{1}

The integrable Teich\-m\"ul\-ler space has recently garnered interest as a parameter space for Weil--Petersson curves
(see \cite{Bi}). We have made attempts to extend the theory, including generalizing the exponent of integrability and incorporating discussions from harmonic analysis into the complex-analytic Teich\-m\"ul\-ler space theory. In this paper, we explore a direction of research that more clearly organizes the existing arguments through the method of simultaneous uniformization of curves from the perspective of quasiconformal Teich\-m\"ul\-ler theory.

Integrable Teich\-m\"ul\-ler spaces $T_p$ have been formulated as subspaces of the universal Teich\-m\"ul\-ler space $T$, extending the definition of the Weil--Petersson metric on the Teich\-m\"ul\-ler space of compact Riemann surfaces. This forms the basis for the complex analytic theory of Weil--Petersson curves. Notable studies include Cui's research on the complex structure of the integrable Teich\-m\"ul\-ler space \cite{Cu}, Takhtajan and Teo's study of the curvature of the Weil--Petersson metric \cite{TT}, and Shen's research on quasiconformal Teich\-m\"ul\-ler space theory \cite{Sh18}.

The universal Teich\-m\"ul\-ler space $T$ can be defined based on the complex dilatation $\mu$ of quasiconformal homeomorphisms, while the integrable Teich\-m\"ul\-ler space $T_p$ is defined by imposing $L_p$-integrability on $\mu$ with respect to the hyperbolic metric. Originally, the theory was for the case $p=2$, but research on generalizing the exponent is also progressing. Extensions to $p \geq 2$ were made by Tang and Shen \cite{TS}, and to $p > 1$ by Wei and Matsuzaki \cite{WM-3}. Furthermore, in \cite{WM-1,WM-4}, the complex structure through the Schwarzian derivative has been given to $T_p$ for all $p \geq 1$. This paper also focuses on the methods for characterizing quasisymmetric mappings of $T_1$ through pre-Schwarzian derivatives.

The universal Teich\-m\"ul\-ler space $T$ is identified with the set of all quasisymmetric homeo\-morphisms $h:\mathbb R \to \mathbb R$ under appropriate normalization. However, they are not necessarily absolutely continuous. To conduct a finer analysis, better regularity is required. Teich\-m\"ul\-ler spaces consisting of locally absolutely continuous quasisymmetric homeomorphisms are collectively called absolutely continuous Teich\-m\"ul\-ler spaces. The BMO Teich\-m\"ul\-ler space, introduced by Astala and Zinsmeister \cite{AZ}, is a typical example that includes various absolutely continuous Teich\-m\"ul\-ler spaces such as $T_p$, and has been well-organized through research by Shen and Wei \cite{SWei}. Current research trends include investigating problems in real-analysis within the framework of Teich\-m\"ul\-ler space theory and vice versa, which can also be applied to the study of integrable Teich\-m\"ul\-ler spaces $T_p$.

In the theory of absolutely continuous Teich\-m\"ul\-ler spaces, we analyze the function spaces to which $\log h'$ for a quasisymmetric homeomorphism $h$ belong. In this paper, we first consider this problem for $T_p$. A special type of the Besov space $\widehat B_p(\mathbb R)$, which is a kind of fractional dimensional Sobolev space on $\mathbb R$, serves as this target space. By generalizing the previous results for $p > 1$, we establish the following characterization of $T_p$ for all $p \geq 1$. Here, ${\rm Re}\,\widehat B_p(\mathbb R)$ stands for the real subspace of $\widehat B_p(\mathbb R)$ consisting of all real-valued functions.

\begin{theorem}\label{theorem10}
For $p \geq 1$, a quasisymmetric homeomorphism $h: \mathbb R \to \mathbb R$ is an element of $T_p$ if and only if $h$ is locally absolutely continuous and $\log h'$ belongs to ${\rm Re}\,\widehat B_p(\mathbb R)$.
\end{theorem}

There are several difficulties in proving this result in the case $p=1$ compared with the case $p>1$: (i) The Besov space $\widehat B_p(\mathbb R)$ for $p \geq 1$ should be more involved than $B_p(\mathbb R)$ for $p>1$; (ii) The boundedness of the composition operator $C_h$ on $\widehat B_p(\mathbb R)$, defined by $C_h(\phi)=\phi \circ h$, is newly presented in the case $p=1$ by using a different interpolation of Banach spaces from the case $p>1$; (iii) To prove the `if' part of the statement for $p=1$, the canonical method of quasiconformal extension of $h$ does not work because of the complicated definition of $\widehat B_p(\mathbb R)$.

Section \ref{3} addresses (i) by defining $\widehat B_p(\mathbb R)$, and Sections \ref{5} and \ref{6} address (ii) by showing the boundedness of $C_h$. The `only-if' part of Theorem \ref{theorem10} is proved in Theorem \ref{main1}. Due to difficulty (iii), the `if' part  
is proved implicitly by showing that the correspondence $h \mapsto \log h'$ is surjective from $T_p$ onto ${\rm Re}\,\widehat B_p(\mathbb R)$. The proof is completed in Section \ref{8}.

A basic method underlying the entire argument for Theorem \ref{theorem10} is conformal welding together with pre-Bers embedding and simultaneous uniformization, which are cross-sections of the same nature of Teich\-m\"ul\-ler spaces. By the conformal welding, any quasisymmetric homeomorphism $h$ of $\mathbb R$ can be represented as the difference of the boundary extensions of conformal homeomorphisms of the upper half-plane $\mathbb H^+$ and the lower half-plane $\mathbb H^-$. Moreover, if $h \in T_p$, then such conformal homeomorphisms extend quasiconformally to $\mathbb C$ with their complex dilatations in $M_p(\mathbb H^\pm)$, the spaces of $p$-integrable Beltrami coefficients.

Conversely, the conformal homeomorphisms $F^\mu$ of $\mathbb H^-$ for $\mu \in M_p(\mathbb H^+)$ represent the elements of $T_p^+$, the Teich\-m\"ul\-ler space defined by $M_p(\mathbb H^+)$. Here, $\log (F^\mu)'$ belongs to the analytic Besov space $\widehat{\mathcal B}_p(\mathbb H^-)$, the boundary extension of whose elements are in $\widehat B_p(\mathbb R)$. Then, we obtain the pre-Bers embedding $\beta:T_p^+ \to \widehat{\mathcal B}_p(\mathbb H^-)$. We note that for the Bers embedding given by the Schwarzian derivative of $F^\mu$, there is no difficulty in setting the appropriate complex Banach space of holomorphic functions in which $T_p$ is embedded, but for the pre-Bers embedding, the definition of $\widehat{\mathcal B}_p(\mathbb H^-)$ for $p=1$ is more involved like $\widehat B_p(\mathbb R)$. This is conducted in Section \ref{4}.

The restriction $\gamma=G|_{\mathbb R}$ of a quasiconformal self-homeomorphism $G$ to $\mathbb R$ is called a quasisymmetric embedding. As a special quasisymmetric embedding, if the complex dilatations of $G$ on $\mathbb H^\pm$ are in $M_p(\mathbb H^\pm)$, we call $\gamma:\mathbb R \to \mathbb C$ a $p$-Weil--Petersson embedding. A Weil--Petersson curve is its image. The idea of the simultaneous uniformization of quasi-Fuchsian groups also works for quasicircles (images of quasisymmetric embeddings), and thus in the same way the space of all $p$-Weil--Petersson embeddings is parametrized by the product of the Teich\-m\"ul\-ler spaces $T_p^+ \times T_p^-$, which we call the Bers coordinates. Then, by generalizing the `if' part of Theorem \ref{theorem10}, we see that $\log \gamma'$ belongs to $\widehat B_p(\mathbb R)$ for a $p$-Weil--Petersson embedding $\gamma=\gamma([\mu^+],[\mu^-])$ determined by $([\mu^+],[\mu^-]) \in T_p^+ \times T_p^-$. In Section \ref{7}, we prepare these arguments for the next two sections. Naturally, this correspondence 
$$
\Lambda:T_p^+ \times T_p^- \to \widehat B_p(\mathbb R), \quad \gamma([\mu^+],[\mu^-])  \mapsto \log \gamma'
$$ 
is a holomorphic injection.

By the pre-Bers embedding $\beta^{\pm}:T_p^\pm \to \widehat{\mathcal B}_p(\mathbb H^\mp)$, the tangent space of the Bers coordinates $T_p^+ \times T_p^-$ may be regarded as $\widehat{\mathcal B}_p(\mathbb H^-) \oplus \widehat{\mathcal B}_p(\mathbb H^+)$. On the other hand, the Szeg\"o projections 
$$
(P^\pm\phi)(z)=\frac{1}{\pi}\int_{\mathbb R} \frac{\phi(t)}{z-t}dt \qquad (z \in \mathbb H^\pm)
$$ 
in $\widehat B_p(\mathbb R)$ are surjective onto $\widehat{\mathcal B}_p(\mathbb H^\pm)$, and under the identification of $\widehat{\mathcal B}_p(\mathbb H^\pm)$ with the closed subspaces of $\widehat B_p(\mathbb R)$ by taking the boundary extension to $\mathbb R$, they yield the associated direct sum decomposition $\widehat B_p(\mathbb R)=\widehat{\mathcal B}_p(\mathbb H^-) \oplus \widehat{\mathcal B}_p(\mathbb H^+)$.

The derivative $d_{([0],[0])}\Lambda$ at the origin of $T_p^+ \times T_p^-$ is the identity map 
$\widehat{\mathcal B}_p(\mathbb H^-) \oplus \widehat{\mathcal B}_p(\mathbb H^+) \to \widehat{\mathcal B}_p(\mathbb H^-) \oplus \widehat{\mathcal B}_p(\mathbb H^+)$. From this, we see that the image of $\Lambda$ contains ${\rm Re}\,\widehat B_p(\mathbb R)$, which proves the `if' part of Theorem \ref{theorem10}. The derivative $d_{([\mu^+],[\mu^-])}\Lambda$ at any point of $T_p^+ \times T_p^-$ maps $\widehat{\mathcal B}_p(\mathbb H^-)$ to $C_{h^-}(\widehat{\mathcal B}_p(\mathbb H^-))$ and $\widehat{\mathcal B}_p(\mathbb H^+)$ to $C_{h^+}(\widehat{\mathcal B}_p(\mathbb H^+))$, where $C_{h^\pm}$ are the composition operators induced by the quasisymmetric homeomorphisms $h^\pm=[\mu^\pm]$ in $T_p^\pm$. Then, we have the surjectivity of the derivative of $\Lambda$ by showing that 
$$
\widehat B_p(\mathbb R)=C_{h^-}(\widehat{\mathcal B}_p(\mathbb H^-)) \oplus C_{h^+}(\widehat{\mathcal B}_p(\mathbb H^+)).
$$
As a consequence, we have the following crucial property of $\Lambda$. This is demonstrated in Theorem \ref{main2} and Corollary \ref{structure} of Section \ref{9}.

\begin{theorem}\label{theorem20}
The map $\Lambda:T_p^+ \times T_p^- \to \widehat B_p(\mathbb R)$ is a
biholomorphic homeo\-morphism onto its image
for $p \geq 1$. Moreover, the composition of $\Lambda$ with the canonical symmetric embedding $T_p \to T_p^+ \times T_p^-$ yields a real-analytic diffeomorphism of $T_p$ onto ${\rm Re}\,\widehat B_p(\mathbb R)$.
\end{theorem}

Therefore, the integrable Teich\-m\"ul\-ler space $T_p$ for every
$p \geq 1$ can be identified with the real Banach space ${\rm Re}\,\widehat B_p(\mathbb R)$
with respect to the real-analytic structure.

This biholomorphic homeomorphism $\Lambda$ also works effectively for the investigation of certain problems in real analysis concerning Weil--Petersson curves. For an overview of real analytic studies of bi-Lipschitz embeddings of $\mathbb R$ onto chord-arc curves, Coifman and Meyer \cite{CM0} and Semmes \cite{Se} provide comprehensive explanations, which motivate our work. The former focuses on such embeddings using harmonic analytic methods, while the latter utilizes theories of quasiconformal mappings. We extend these studies to the context of Weil--Petersson embeddings.

Our strategy is that by the biholomorphic homeomorphism $\Lambda$ from the Bers coordinates $T_p^+ \times T_p^-$ of the $p$-Weil--Petersson embeddings to the Besov space $\widehat B_p(\mathbb R)$, we relate the derivative of $\Lambda$ at $([\mu^+],[\mu^-]) \in T_p^+ \times T_p^-$ to the Cauchy projection with respect to the Weil--Petersson curve $\Gamma=\gamma(\mathbb R)$ for $\gamma=\gamma([\mu^+],[\mu^-])$. 
Let $\Omega^\pm$ be the domains in $\mathbb C$ separated by $\Gamma$. The Cauchy integrals 
$$
({P}^\pm_\Gamma \psi)(\zeta)=\frac{1}{\pi}\int_{\Gamma} \frac{\psi(z)}{\zeta-z}dz \quad (\zeta \in \Omega^\pm)
$$
of an integrable function $\psi$ on $\Gamma$ yield holomorphic functions on $\Omega^\pm$. We consider the Besov space 
$\widehat B_p(\gamma(\mathbb R))$ on $\Gamma$ as the push-forward of $\widehat B_p(\mathbb R)$ by $\gamma$ 
and the analytic Besov spaces $\widehat {\mathcal B}_p(\Omega^\pm)$ on $\Omega^\pm$ 
as the push-forward of $\widehat {\mathcal B}_p(\mathbb H^\pm)$ by the normalized Riemann mappings $\mathbb H^\pm \to \Omega^\pm$. Then, the Cauchy integrals induce the projections $P_\Gamma^\pm:\widehat B_p(\gamma(\mathbb R)) \to \widehat {\mathcal B}_p(\Omega^\pm)$. By taking the boundary extension, we assume that $\widehat {\mathcal B}_p(\Omega^\pm)$ are closed subspaces of $\widehat B_p(\gamma(\mathbb R))$. 

On the other hand, the isomorphic derivative $d_{([\mu^+],[\mu^-])}\Lambda$ of the biholomorphic homeo\-morphism $\Lambda$ conjugates the projections of the tangent space of $T_p^+ \times T_p^-$ at $([\mu^+],[\mu^-])$ onto both factors with the projections regarding the direct sum decomposition 
$\widehat{B}_p(\mathbb R)=C_{h^-}(\widehat{\mathcal B}_p(\mathbb H^-)) \oplus C_{h^+}(\widehat{\mathcal B}_p(\mathbb H^+))$, which are
$$
P_{([\mu^+],[\mu^-])}^\pm:\widehat{B}_p(\mathbb R) \to C_{h^\pm}(\widehat{\mathcal B}_p(\mathbb H^\pm)).
$$
Then, a novel aspect of our argument lies in the assertion that the Cauchy projections ${P}^\pm_\Gamma$ are 
the conjugates of $P_{([\mu^+],[\mu^-])}^\pm$ by the push-forward $\gamma_*$. By this relationship, we can reduce the properties of the Cauchy projections ${P}^\pm_\Gamma$ for the embedding $\gamma=\gamma([\mu^+],[\mu^-])$ to the derivative $d_{([\mu^+],[\mu^-])}\Lambda$ of the biholomorphic homeomorphism.

One of the consequences from this argument
is as follows, which is given in Theorem \ref{Cauchy} and Corollary \ref{Calderon} of Section \ref{10}. 
This problem originates in Cardel\'on's work.

\begin{theorem}\label{theorem30}
In the Besov space $\widehat B_p(\gamma(\mathbb R))$ on the $p$-Weil--Petersson curve $\Gamma=\gamma(\mathbb R)$
of $\gamma=\gamma([\mu^+],[\mu^-])$ for $p \geq 1$, the Cauchy projections $P^\pm_\Gamma$ 
coincide with $\gamma_* \circ P^\pm_{([\mu^+],[\mu^-])} \circ \gamma_*^{-1}$,
which
are bounded linear projections onto $\widehat {\mathcal B}_p(\Omega^\pm)$ associated with the topological direct sum decomposition
$$
\widehat B_p(\gamma(\mathbb R)) = \widehat {\mathcal B}_p(\Omega^+) \oplus \widehat {\mathcal B}_p(\Omega^-).
$$
\end{theorem}

A significant property of the Cauchy projections in the analysis on the curves is, however, its analytic dependence and the estimate of the operator norm as the curves vary (see \cite{CM0, Se}). In our method, it is clear that $\Vert P^\pm_\Gamma \Vert$ can be estimated in terms of $\Vert C_{h^\pm} \Vert$ for the composition operators $C_{h^\pm}$. In addition, since $P^\pm_\Gamma$ is related to the derivative of the biholomorphic homeomorphism $\Lambda$, its holomorphic dependence is also clear.

\begin{theorem}\label{theorem40}
For $p \geq 1$, the conjugates of the Cauchy projections ${P^{\pm}}_{([\mu^+],[\mu^-])}$ acting on $\widehat B_p(\mathbb R)$ depend holomorphically on $([\mu^+],[\mu^-]) \in T_p^+ \times T_p^-$ with respect to the operator norm topology.
\end{theorem}

An equivalent statement for the conjugate of the Cauchy transform ${\mathcal H}_\Gamma=-i(P_\Gamma^+-P_\Gamma^-)$ is given in Theorem \ref{conjugateH} of Section \ref{11}. We also exhibit an important application of this result after Coifman and Meyer \cite{CM}.

\section{Integrable Teich\-m\"ul\-ler spaces}\label{2}

Within the universal Teich\-m\"ul\-ler space, integrable Teich\-m\"ul\-ler spaces are defined by the integrability of Beltrami coefficients. One of the fundamental problems in integrable Teich\-m\"ul\-ler spaces is characterizing these spaces through quasisymmetric homeomorphisms, which are the boundary mappings of quasiconformal homeomorphisms. In this section, to introduce this problem, we lay the foundations of integrable Teich\-m\"ul\-ler spaces.

An orientation-preserving homeomorphism $h:\mathbb R \to \mathbb R$ is called a 
{\it quasisymmetric} homeo\-morphism 
if it satisfies the following doubling condition: for any bounded interval $I \subset \mathbb R$, 
there exists a constant $M>1$ (called a doubling constant) such that $|h(I)| \leq M|h(\frac{1}{2}I)|$, 
where $|\cdot|$ denotes the Lebesgue measure 
on $\mathbb R$ and $\frac{1}{2}I$ represents the interval with the same center as $I$ but half the length. 
An orientation-preserving homeomorphism $H$ defined on a planar region is 
called a {\it quasiconformal} homeomorphism if $H$ has locally integrable partial derivatives $H_z, H_{\bar z}$
in the distribution sense, 
and the {\it complex dilatation} $\mu(z) = H_{\bar z}/H_z$ satisfies $\Vert \mu \Vert_\infty < 1$.

Any quasiconformal homeomorphism $H:\mathbb H \to \mathbb H$ ($\mathbb H$ is either 
the upper half-plane $\mathbb H^+$ or the lower half-plane $\mathbb H^-$)
can be continuously extended to a quasisymmetric homeomorphism $h:\mathbb R \to \mathbb R$. 
Conversely, any quasisymmetric homeomorphism $h$ is a continuous extension of such a quasiconformal homeomorphism $H$.

A complex-valued measurable function $\mu$ on $\mathbb H$ that satisfies $\Vert \mu \Vert_\infty < 1$ 
is called a {\it Belt\-rami coefficient}. The space of all Beltrami coefficients on $\mathbb H$ is denoted by $M(\mathbb H)$.
For a Belt\-rami coefficient $\mu \in M(\mathbb H)$, there uniquely exists a quasiconformal self-homeo\-morphism 
$H(\mu)$ of $\mathbb H$ with the complex dilatation $\mu$ that satisfies the normalization condition (which means
fixing $0,1,\infty$ in all cases also for other mappings appearing in this paper)
by the measurable Riemann mapping theorem. Therefore, the entirety of such quasiconformal homeomorphisms $H(\mu)$ can be 
identified with $M(\mathbb H)$.

The set of all quasisymmetric homeomorphisms $h:\mathbb R \to \mathbb R$ that 
satisfy the normalization conditions is called the {\it universal Teich\-m\"ul\-ler space}, which is denoted by $T$. 
The correspondence through boundary extension from quasiconformal homeomorphisms $H(\mu)$ of 
$\mathbb H$ to quasisymmetric homeomorphisms $h(\mu)$ of $\mathbb R$ defines the {\it Teich\-m\"ul\-ler projection}
$\pi:M(\mathbb H) \to T$.
The equivalence class $[\mu]$ of Beltrami coefficients $\mu$ relative to $\pi$ is called a Teich\-m\"ul\-ler equivalence class.

The topology of $T$ is induced from $M(\mathbb H)$. 
More precisely, it is the quotient topology determined by $\pi$.
In another way, it can be given as the underlying topology of the Teich\-m\"ul\-ler distance 
defined by the $L_\infty$-norm on $M(\mathbb H)$.
It is well-known that they are equivalent.

The complex structure and the metric structure of $T$
are induced through the embedding of $T$ into a certain complex Banach space.
For $\mu \in M(\mathbb H^+)$, let $F^\mu$ be a quasiconformal self-homeomorphism of $\mathbb C$ satisfying the normalization condition, having the complex dilatation $\mu$ in the upper half-plane $\mathbb H^+$, 
and being conformal in the lower half-plane $\mathbb H^-$. 
For Beltrami coefficients $\mu$ and $\nu$,
we see that $\pi(\mu)=\pi(\nu)$ if and only if $F^\mu|_{\mathbb H^-}=F^\nu|_{\mathbb H^-}$. 
By taking the Schwarzian derivative 
$S_F=(\log F')''-\frac{1}{2}((\log F')')^2$ 
of $F=F^\mu|_{\mathbb H^-}$, $T$ can be topologically embedded into the complex Banach space of ${\mathcal A}_\infty(\mathbb H)$ of holomorphic functions $\Psi$ on $\mathbb H=\mathbb H^-$
with the hyperbolic $L_\infty$-norm 
$\Vert \Psi \Vert_{\mathcal A_\infty}=\sup_{z \in \mathbb H}\,|({\rm Im}\,z)^2\Psi(z)|$.
This mapping $\alpha:T \to {\mathcal A}_\infty(\mathbb H)$ well defined by
$[\mu] \mapsto S_F$ is called the {\it Bers embedding}. The image $\alpha(T)$ is a bounded domain in ${\mathcal A}_\infty(\mathbb H)$,
which is identified with $T$. 

Moreover, $T$ has the structure of a group with the composition of normalized quasisymmetric homeomorphisms as its operation. 
For $h(\mu)=[\mu]$ and $h(\nu)=[\nu]$,
the composition $h(\mu) \circ h(\nu)$ is denoted by $[\mu]\ast [\nu]$ and
the inverse $h(\mu)^{-1}$ is denoted by $[\mu]^{-1}$.
In fact, 
the composition of normalized quasiconformal self-homeomorphisms $H(\mu)$ and $H(\nu)$ of $\mathbb H$ gives the group structure 
$\mu \ast \nu$ on $M(\mathbb H)$ and the Teich\-m\"ul\-ler projection $\pi$ is a homomorphism as $\pi(\mu \ast \nu)=[\mu] \ast [\nu]$.
The right translation $r_{[\nu]}:T \to T$ for any $[\nu] \in T$ is defined as $[\mu] \mapsto [\mu] \ast [\nu]$, providing a biholomorphic automorphism of $T$. However, $T$ is not a topological group.

For literature on quasiconformal mappings and the universal Teich\-m\"ul\-ler space, useful references include \cite{Ah, Le}.
Proofs of the aforementioned facts can be found therein.

We introduce the integrable Teich\-m\"ul\-ler space $T_p$ in $T$.

\begin{definition}
For $p \geq 1$, a Beltrami coefficient $\mu \in M(\mathbb H)$ is $p$-integrable 
with respect to the hyperbolic metric if it satisfies
$$
\Vert \mu \Vert_p=\left(\int_{\mathbb H} |\mu(z)|^p \frac{dxdy}{|{\rm Im}\,z|^2}\right)^{1/p}<\infty.
$$
The set of all such $p$-integrable Beltrami coefficients is denoted by $M_p(\mathbb H)$.
The {\it $p$-integrable Teich\-m\"ul\-ler space} (also referred to as Weil--Petersson Teich\-m\"ul\-ler space) is defined as $T_p = \pi(M_p(\mathbb H))$ for the Teich\-m\"ul\-ler projection $\pi$. 
\end{definition}

It follows immediately from this definition that $T_p \subset T_q$ if $p<q$.

Since the universal Teich\-m\"ul\-ler space $T$ is defined as the set of
all normalized quasisymmetric homeomorphisms, the subset $T_p \subset T$ is also regarded as
a certain subset of normalized quasisymmetric homeomorphisms.
We consider this characterization later.

The topology of the integrable Teich\-m\"ul\-ler space $T_p$ is induced from $M_p(\mathbb H)$
with norm $\Vert \mu \Vert_p+\Vert \mu \Vert_\infty$, which is 
the quotient topology by $\pi$.
This can be also given as the underlying topology of the Weil--Petersson distance 
defined by the $L_p$-norm on $M_p(\mathbb H)$, which are known to be equivalent (see \cite[Proposition 5.2]{WM-*}).

The complex structure and the metric structure of $T_p$ for $p \geq 1$ are also defined 
through the Bers embedding $\alpha:T \to A_\infty(\mathbb H)$. The image of $T_p$ under $\alpha$ is contained in
the complex Banach space of holomorphic functions
$$
\mathcal A_p(\mathbb H)=\{\Psi \in \mathcal A_\infty(\mathbb H) \mid \Vert \Psi \Vert_{\mathcal A_p}=
\left(\int_{\mathbb H} |({\rm Im}\,z)^2\Psi(z)|^p \frac{dxdy}{|{\rm Im}\,z|^2}\right)^{1/p}<\infty\}.
$$
We note that the inclusion $\mathcal A_p(\mathbb H) \hookrightarrow \mathcal A_\infty(\mathbb H)$ is continuous, and 
$\mathcal A_p(\mathbb H) \hookrightarrow \mathcal A_q(\mathbb H)$ is continuous if $p<q$.
Moreover, $\alpha$ gives a topological embedding of $T_p$ respecting these stronger topologies
whose image is a connected open subset of $\mathcal A_p(\mathbb H)$.
In this way, the complex Banach structure is provided for $T_p$.
These facts for all $p \geq 1$ are shown in \cite{WM-1,WM-*}. Concerning the Weil--Petersson metric defined through
the Bers embedding, one can refer to \cite[Section 5]{WM-*}.

The {\it BMO Teich\-m\"ul\-ler space} $T_B$ mentioned in the introduction (see \cite{AZ, SWei})
is identified with the set of all normalized strongly quasisymmetric homeomorphisms.
Here, an orientation-preserving homeomorphism $h:\mathbb R \to \mathbb R$ is 
called {\it strongly quasisymmetric} if there exist positive constants $C, \delta>0$ such that
for any bounded interval $I \subset \mathbb R$ 
and any measurable subset $E \subset I$, it holds that
$|h(E)|/|h(I)| \leq C (|E|/|I|)^\delta$. 
This condition is equivalent to saying that $h$ is locally absolutely continuous and $h'$ is an $A_\infty$-weight (see \cite{CF}). 

The BMO Teich\-m\"ul\-ler space $T_B$ is also represented by Beltrami coefficients. Namely, we define $\mu \in M(\mathbb H)$ to be in
a subset $M_B(\mathbb H)$ if it satisfies the {\it Carleson measure condition}
$$
\Vert \mu \Vert_{\ast}^2=\sup_{I \subset \mathbb R}\frac{1}{|I|}\int_{\widehat I}|\mu(z)|^2\frac{dxdy}{|{\rm Im}\,z|}<\infty,
$$
where the supremum is taken over all bounded intervals $I$ on $\mathbb{R}$ and $\widehat I$ is the square in $\mathbb H$
over $I$. Then, $T_B=\pi(M_B(\mathbb H))$. 
It is known that $T_p$ is contained in $T_B$ for every $p \geq 1$ (see \cite[Proposition 2.2]{WM-3}).
Hence, any quasisymmetric homeomorphism $h \in T_p$ is strongly quasisymmetric, and in particular, it is 
locally absolutely continuous. We consider the problem of determining the function space on $\mathbb R$ to which $\log h'$ belongs.

\section{Besov Spaces}\label{3}

We set up the function space on $\mathbb R$ to which $\log h'$ belongs for a quasisymmetric homeo\-morphism $h \in T_p$.
For $m \in \mathbb{N}$ and $t \in \mathbb{R}$, the $m$-th order difference of a function $\phi$ on $\mathbb{R}$ is defined as
$$
\Delta_t^1 \phi(x) = \phi(x+t) - \phi(x); \quad
\Delta_t^{m+1} \phi(x) = \Delta_t^m \phi(x+t) - \Delta_t^m \phi(x).
$$
For $s \in \mathbb{R}$ and $0 < p,\, q \leq \infty$, a semi-norm of a locally integrable complex-valued function $\phi$ on $\mathbb{R}$ is given by
$$
\Vert \phi \Vert_{\dot B^s_{p,q}} = \left(\int_{-\infty}^{\infty}
|t|^{-sq} \Vert \Delta_t^{\lfloor s \rfloor+1} \phi \Vert_{L_p}^q \frac{dt}{|t|}
\right)^{1/q}.
$$
The set of those $\phi$ with $\Vert \phi \Vert_{\dot B^s_{p,q}} < \infty$ is referred to as the {\it homogeneous Besov space}.

This is the classical definition of Besov spaces, but in relation to integrable Teich\-m\"ul\-ler spaces, only the 
critical case $p=q=1/s$ is considered. Therefore, we provide the following specific definition.

\begin{definition}
For $p > 1$, a locally integrable complex-valued function $\phi$ on $\mathbb{R}$ belongs to the {\it Besov space} $B_p(\mathbb{R})$ if
$$
\Vert \phi \Vert_{B_p} = \left(\int_{-\infty}^\infty \int_{-\infty}^\infty \frac{|\phi(x+t)-\phi(x)|^p}{t^2} \, dx \, dt \right)^{1/p}
$$
is finite. Moreover, for $p \geq 1$,
$$
\Vert \phi \Vert_{B_p^{\#}} = \left(\int_{-\infty}^\infty \int_{-\infty}^\infty 
\frac{|\phi(x+2t)-2\phi(x+t)+\phi(x)|^p}{t^2} \, dx \, dt \right)^{1/p}
$$
defines the set $B_p^{\#}(\mathbb R)$ consisting of those $\phi$ for which this quantity is finite.
\end{definition}

\begin{remark}
For $p = 1$, functions $\phi$ such that $\Vert \phi \Vert_{B_p} < \infty$ are restricted to constant functions.
See \cite{Br} as noted in \cite[Exercise 17.14]{Leo}.
\end{remark}

We see that $B_p(\mathbb{R}) \hookrightarrow B_q(\mathbb{R})$ and $B_p^{\#}(\mathbb{R}) \hookrightarrow B_q^{\#}(\mathbb{R})$
are continuous for $p<q$ by the Sobolev embedding theorem (see \cite[Theorem 2.14]{Sa}).

A locally integrable complex-valued function $\phi$ on $\mathbb{R}$ is said to be {\it BMO} if
$$
\Vert \phi \Vert_{\rm BMO} = \sup_{I \subset \mathbb R} \frac{1}{|I|} \int_I |\phi(x) - \phi_I| \, dx < \infty,
$$
where the supremum is taken over all bounded intervals $I$ on $\mathbb{R}$, and $\phi_I$ denotes the integral mean of $\phi$ over $I$. The set of all BMO functions is denoted by ${\rm BMO}(\mathbb R)$. It is easy to see that 
$\Vert \phi \Vert_{\rm BMO}
\leq \Vert \phi \Vert_{L_\infty}$. It is known that $\log h'$ for
a strongly quasisymmetric homeomorphism $h$ belongs to ${\rm BMO}(\mathbb R)$ since $h'$ is an $A_\infty$-weight.

\begin{definition}
For $p \geq 1$, we define the function space $\widehat B_p(\mathbb R)$ as
$$
\widehat B_p(\mathbb R) = B_p^{\#}(\mathbb R) \cap {\rm BMO}(\mathbb R),
$$
with the semi-norm $\Vert \phi \Vert_{{\widehat B}_p} = \Vert \phi \Vert_{B_p^{\#}} + \Vert \phi \Vert_{\rm BMO}$.
\end{definition}

For $p > 1$, the semi-norms $\Vert \phi \Vert_{B_p}$ and $\Vert \phi \Vert_{{\widehat B}_p}$ are equivalent. 
Indeed, $\Vert \phi \Vert_{B_p} \asymp \Vert \phi \Vert_{B_p^{\#}}$ is well-known (see \cite[Proposition 17.17]{Leo})
and $\Vert \phi \Vert_{\rm BMO} \leq \Vert \phi \Vert_{B_p}$ can be found in \cite[Proposition 2.2]{WM-3}.

The reason for adding the BMO norm is twofold: firstly, to restrict indeterminacy to only a constant function difference, and secondly, to ensure that when defining Besov norms for functions on the unit circle $\mathbb{S}$ in the same manner as below, they correspond isomorphically through the Cayley transformation.

By ignoring differences in constant functions, $\Vert \cdot \Vert_{B_p}$ and 
$\Vert \cdot \Vert_{{\widehat B}_p}$ can be treated as norms, allowing $B_p(\mathbb R)$ and 
$\widehat B_p(\mathbb R)$ to be regarded as complex Banach spaces.

On the unit circle $\mathbb S$, we define the spaces $B_p(\mathbb S)$, $B_p^{\#}(\mathbb S)$, 
and $\widehat{B}_p(\mathbb S)$ of integrable functions in the same way. 
For the Cayley transformation $K(x)=(x-i)/(x+i)$ which maps
$\widehat{\mathbb R}= \mathbb R \cup \{\infty\}$ onto $\mathbb S$ with $K(\infty)=1$,
let $K_*(\phi)=\phi_*$ be the push-forward of a function $\phi$ on $\mathbb R$ to that on $\mathbb S$;
$K_*(\phi)=\phi \circ K^{-1}=\phi_*$. Then, $K_*$ defines the isometric isomorphism from
$B_p(\mathbb R)$ onto $B_p(\mathbb S)$ for $p>1$ (see \cite[p.8]{WM-4}).
We also note that 
$\Vert K_*(\phi) \Vert_{\rm BMO} \asymp \Vert \phi \Vert_{\rm BMO}$ (see \cite[Corollary VI.1.3]{Ga}).
Moreover, we see that $K_*$ gives a Banach isomorphism from $\widehat{B}_p(\mathbb R)$ onto 
$\widehat{B}_p(\mathbb S)$ for $p > 1$.
Indeed,
$$
\Vert \phi \Vert_{\widehat B_p} \asymp \Vert \phi \Vert_{B_p}=\Vert K_*(\phi) \Vert_{B_p}
\asymp \Vert K_*(\phi) \Vert_{\widehat B_p}.
$$
In the case $p=1$, we can obtain the same result, 
but the proof is more complicated and carried out later at the end of Section \ref{4}.

\begin{proposition}\label{isomorphism2}
$K_*:\widehat{B}_p(\mathbb R) \to \widehat{B}_p(\mathbb S)$ is a Banach isomorphism for $p \geq 1$. 
\end{proposition}

The homogeneous Sobolev space ${\dot W}^1_1(\mathbb R)$ consists of locally integrable complex valued functions 
$\phi$ on $\mathbb R$ such that the distributional derivative $\phi'$ is integrable on $\mathbb R$.
The semi-norm is given by $\Vert \phi \Vert_{{\dot W}^1_1}=\Vert \phi' \Vert_{L_1}$.
We have $\Vert K_*(\phi) \Vert_{{\dot W}^1_1}= \Vert \phi \Vert_{{\dot W}^1_1}$ by the change of variables.
The following result is used later.

\begin{lemma}\label{W1}
The semi-norm $\Vert \phi \Vert_{\widehat B_1}=\Vert \phi \Vert_{B_1^{\#}}+\Vert \phi \Vert_{\rm BMO}$ on 
$\widehat B_1(\mathbb R)$ is equivalent to $\Vert \phi \Vert_{B_1^{\#}}+\Vert \phi \Vert_{{\dot W}^1_1}$.
In particular, 
$\widehat B_1(\mathbb R)=B_1^{\#}(\mathbb R) \cap {\dot W}^1_1(\mathbb R)$.
The same norm estimate holds true for $\widehat B_1(\mathbb S)$, but $\widehat B_1(\mathbb S)=B_1^{\#}(\mathbb S)$
in the unit circle case.
\end{lemma}

\begin{proof}
We first prove the statement for $\widehat B_1(\mathbb S)$, and then transfer it to $\widehat B_1(\mathbb R)$.
On the unit circle $\mathbb S$, $B_1^{\#}(\mathbb S)$ is contained in ${\dot W}^1_1(\mathbb S)$. 
Indeed, if $\phi_* \in B_1^{\#}(\mathbb S)$, then $\phi_*-\alpha$ also belongs to $L_1(\mathbb S)$ 
for every $\alpha \in \mathbb C$
by compactness of $\mathbb S$, and it follows from \cite[Theorem 17.66]{Leo} that
$$
\Vert \phi_* \Vert_{{\dot W}^1_1}\lesssim
\Vert \phi_* \Vert_{B_1^{\#}}+\inf_{\alpha \in \mathbb C}\Vert \phi_*-\alpha \Vert_{L_1}<\infty.
$$
In addition, this is bounded by 
$\Vert \phi_* \Vert_{B_1^{\#}}+2\pi \Vert \phi_* \Vert_{\rm BMO} \lesssim \Vert \phi_* \Vert_{\widehat B_1}$.
Therefore, $\Vert \phi_* \Vert_{B_1^{\#}}+\Vert \phi_* \Vert_{{\dot W}^1_1} \lesssim \Vert \phi_* \Vert_{\widehat B_1}$
is satisfied. The converse estimate follows from $\Vert \phi_* \Vert_{\rm BMO} \leq \Vert \phi_* \Vert_{L_\infty}$ and
$$
\inf_{\alpha \in \mathbb C}\Vert \phi_*-\alpha \Vert_{L_\infty} \leq \Vert \phi_* \Vert_{{\dot W}^1_1}
$$ 
which is in the specificity of one-dimensional spaces. 
Thus, $\Vert \phi_* \Vert_{B_1^{\#}}+\Vert \phi_* \Vert_{{\dot W}^1_1} \asymp \Vert \phi_* \Vert_{\widehat B_1}$.

The corresponding result for $\widehat B_1(\mathbb R)$ is true
due to the isomorphism $K_*:\widehat{B}_1(\mathbb R) \to \widehat{B}_1(\mathbb S)$ by Proposition \ref{isomorphism2}
as well as 
$\Vert K_*(\phi) \Vert_{{\dot W}^1_1} = \Vert \phi \Vert_{{\dot W}^1_1}$.
Indeed, these imply
$$
\Vert \phi \Vert_{{\dot W}^1_1} = \Vert \phi_* \Vert_{{\dot W}^1_1} 
\lesssim \Vert \phi_* \Vert_{\widehat B_1} \asymp \Vert \phi \Vert_{\widehat B_1}.
$$
The converse estimate $\Vert \phi \Vert_{\rm BMO} \leq \Vert \phi \Vert_{{\dot W}^1_1}$ is obtained
in the same reason as above.
\end{proof}

\section{The Hilbert transform and the Szeg\"o projection}\label{4}

We introduce holomorphic functions on $\mathbb H$ whose boundary extension to $\mathbb R$ is 
in the Besov space on $\mathbb R$, which are called analytic Besov functions.
Conversely, we see that any Besov function on $\mathbb R$ can be represented as a sum of
the boundary extensions of analytic Besov functions on the upper half-plane $\mathbb H^+$
and the lower half-plane $\mathbb H^-$.

\begin{definition}
For $p > 1$, we define the {\it analytic Besov space} $\mathcal{B}_p(\mathbb{H})$ 
as the set of holomorphic functions $\Phi$ on $\mathbb{H}$ for which the semi-norm
$$
\Vert \Phi \Vert_{\mathcal{B}_p}=\left(\int_{\mathbb{H}}|(\mathrm{Im}\,z)\Phi'(z)|^p
\frac{dxdy}{|\mathrm{Im}\,z|^{2}}\right)^{1/p}
$$
is finite. Moreover, for $p \geq 1$,
$$
\Vert \Phi \Vert_{\mathcal{B}_p^\#}=\left(\int_{\mathbb{H}}|(\mathrm{Im}\,z)^2\Phi''(z)|^p
\frac{dxdy}{|\mathrm{Im}\,z|^{2}}\right)^{1/p}
$$
defines the set $\mathcal{B}_p^{\#}(\mathbb{R})$ consisting of those $\Phi$ for which this quantity is finite.
\end{definition}

\begin{remark}
For $p = 1$, functions $\Phi$ such that $\Vert \Phi \Vert_{\mathcal{B}_p} < \infty$ are restricted to constant functions.
See \cite[p.132]{Zhu}.
\end{remark}

We see that $\mathcal{B}_p(\mathbb{H}) \hookrightarrow \mathcal{B}_q(\mathbb{H})$ and 
$\mathcal{B}_p^{\#}(\mathbb{H}) \hookrightarrow \mathcal{B}_q^{\#}(\mathbb{H})$
are continuous for $p<q$ (see \cite[Propositions 2.1, 2.2]{WM-*}).

A holomorphic function $\Phi$ on $\mathbb{H}$ is {\it BMOA} if it satisfies the Carleson measure condition
$$
\Vert \Phi \Vert_{\mathrm{BMOA}}^2=\sup_{I \subset \mathbb{R}\,}\frac{1}{|I|}
\int_{\widehat{I} \subset \mathbb{H}} |(\mathrm{Im}\,z)\Phi'(z)|^2\frac{dxdy}{|\mathrm{Im}\,z|}<\infty,
$$
where the supremum is taken over all bounded intervals $I$ on $\mathbb{R}$, and $\widehat{I}$ denotes the square in $\mathbb{H}$ with $I$ as its base. The set of all BMOA functions is denoted by $\mathrm{BMOA}(\mathbb{H})$.

\begin{definition}
For $p \geq 1$, we define the function space $\widehat{\mathcal{B}}_p(\mathbb{H})$ as
$$
\widehat{\mathcal{B}}_p(\mathbb{H}) = \mathcal{B}_p^{\#}(\mathbb{H}) \cap \mathrm{BMOA}(\mathbb{H}),
$$
with the semi-norm
$$
\Vert \Phi \Vert_{\widehat{\mathcal{B}}_p}=\Vert \Phi \Vert_{\mathcal{B}_p^{\#}}+\Vert \Phi \Vert_{\mathrm{BMOA}}.
$$
\end{definition}

For $p > 1$, the semi-norms $\Vert \Phi \Vert_{\mathcal{B}_p}$ and $\Vert \Phi \Vert_{\widehat{\mathcal{B}}_p}$ are equivalent. 
This is because $\Vert \Phi \Vert_{\mathcal{B}_p^{\#}} \lesssim \Vert \Phi \Vert_{\mathcal{B}_p}$ and
$\Vert \Phi \Vert_{\mathrm{BMOA}} \lesssim \Vert \Phi \Vert_{\mathcal{B}_p}$ (see \cite[Section 2]{WM-*}).
The rationale for adding the BMOA norm is similar to that in the case of the Besov space norm.

Ignoring differences in constant functions allows $\Vert \cdot \Vert_{\mathcal{B}_p}$ and $\Vert \cdot \Vert_{\widehat{\mathcal{B}}_p}$ to be treated as norms, making it possible to regard $\mathcal{B}_p(\mathbb{H})$ and $\widehat{\mathcal{B}}_p(\mathbb{H})$ as complex Banach spaces.

Here is the relationship between integrable Beltrami coefficients and analytic Besov spaces (see \cite[Theorem 3.11]{WM-*}):

\begin{lemma}\label{conformal}
Let $p \geq 1$.
For a quasiconformal self-homeomorphism $F^{\mu}$ of $\mathbb{C}$, which has complex dilatation $\mu$ in $\mathbb{H}^+$ and is conformal in $\mathbb{H}^-$, the condition $\mu \in M_p(\mathbb{H}^+)$ is equivalent to $\log (F^{\mu}|_{\mathbb{H}^-})' \in \widehat{\mathcal{B}}_p(\mathbb{H}^-)$.
\end{lemma}

Next, we consider the boundary extension of holomorphic functions in the analytic Besov space $\widehat{\mathcal B}_p(\mathbb H)$
for $p \geq 1$. On the half-plane $\mathbb H$ and its boundary $\mathbb R$,
we will see that the trace of $\widehat{\mathcal B}_p(\mathbb H)$
on $\mathbb{R}$ forms a closed subspace of the Besov space $\widehat{B}_p(\mathbb R)$.
The essential issue is the correspondence between ${\mathcal B}_p^{\#}(\mathbb H)$ and $B_p^{\#}(\mathbb R)$
(or between ${\mathcal B}_p(\mathbb H)$ and $B_p(\mathbb R)$ when $p>1$). 

\begin{proposition}\label{boundary}
There are the following relations between the function spaces ${\mathcal B}_p^{\#}(\mathbb H)$ 
and $B_p^{\#}(\mathbb R)$
for $p \geq 1$. The same results hold true between ${\mathcal B}^{\#}_p(\mathbb D)$ and ${B}^{\#}_p(\mathbb S)$.
\begin{enumerate}
\item 
Every holomorphic function $\Phi \in {\mathcal B}_p^{\#}(\mathbb H)$ has non-tangential limits
almost everywhere on $\mathbb R$. This defines the boundary function $\phi$ of $\Phi$ that belongs to 
$B_p^{\#}(\mathbb R)$;
\item 
Conversely, for a measurable function $\phi \in B_p^{\#}(\mathbb R)$ that
is extendable to a holomorphic function on $\mathbb H$, the Poisson integral $\Phi$ of $\phi$ belongs to 
${\mathcal B}_p^{\#}(\mathbb H)$;
\item
In the above correspondence between $\Phi \in {\mathcal B}_p^{\#}(\mathbb H)$ and $\phi \in B_p^{\#}(\mathbb R)$,
their norms are uniformly comparable as $\Vert \Phi \Vert_{{\mathcal B}_p^{\#}} \asymp \Vert \phi \Vert_{B_p^{\#}}$.
\end{enumerate}
\end{proposition}

There are several references verifying these claims.
Among them, we cite \cite[\S V.5.1]{St} for discussions from real analysis in the case of $\mathbb H$, and
\cite[Theorems 2.1 and 5.1]{Pa} from complex analysis in the case of $\mathbb D$. Concerning statement (2), the Poisson integral of any 
$\phi \in {B}_p^{\#}(\mathbb R)$ not necessarily holomorphically extendable can be also characterized by
the integrability of its partial derivatives generalizing the definition of 
${\mathcal B}_p^{\#}(\mathbb H)$.

The other factors in the norms
for $\widehat{\mathcal B}_p(\mathbb H)$ and $\widehat{B}_p(\mathbb R)$
correspond neatly under the boundary extension $\Phi \mapsto \phi$ because 
$\Vert \Phi \Vert_{\rm BMOA}
\asymp \Vert \phi \Vert_{\rm BMO}$ (see \cite[Theorem 9.19]{Zhu}). For $p=1$, we may also use
$\Vert \Phi \Vert_{\dot{H}_1^1}
=\Vert \phi \Vert_{{\dot W}^1_1}$ in account of Lemma \ref{W1} and the corresponding claim \cite[Proposition 2.5]{WM-*}
for $\widehat{\mathcal B}_1(\mathbb H)$ in terms of the Hardy space norm $\Vert \Phi \Vert_{\dot{H}_1^1}=
\Vert \Phi' \Vert_{{H}_1}$.

Consequently, we obtain the following:

\begin{proposition}\label{trace}
Any $\Phi \in \widehat{\mathcal{B}}_p(\mathbb{H})$ has non-tangential limits almost everywhere on $\mathbb{R}$, and its boundary function belongs to $\widehat{B}_p(\mathbb{R})$. This defines the trace operator 
$E: \widehat{\mathcal{B}}_p(\mathbb{H}) \to \widehat{B}_p(\mathbb{R})$, which is a Banach isomorphism onto its image.
\end{proposition}

Considering the both half-planes $\mathbb H^+$ and $\mathbb H^-$,
we obtain the closed subspaces 
$E^+(\widehat{\mathcal B}_p(\mathbb H^+))$ and $E^-(\widehat{\mathcal B}_p(\mathbb H^-))$ 
in $\widehat{B}_p(\mathbb R)$. The functions in
$E^+(\widehat{\mathcal B}_p(\mathbb H^+))$ and $E^-(\widehat{\mathcal B}_p(\mathbb H^-))$ correspond by
complex conjugation.
By the identification
under the Banach isomorphism $E^\pm:\widehat{\mathcal B}_p(\mathbb H^\pm) \to \widehat{B}_p(\mathbb R)$,
we may regard $\widehat{\mathcal B}_p(\mathbb H^\pm)$ are
closed subspaces of $\widehat{B}_p(\mathbb R)$. 

Conversely, the projection associated with $E$ from the Besov space $\widehat{B}_p(\mathbb{R})$ to the analytic Besov space 
$\widehat{\mathcal B}_p(\mathbb H) \cong E(\widehat{\mathcal B}_p(\mathbb H))$ is specifically provided
by using the following map.

\begin{definition}
For a rapidly decreasing function $\phi$ in the Schwartz class on $\mathbb{R}$,
we define the singular integral
$$
\mathcal{H}(\phi)(x) = \text{p.v.}\frac{1}{\pi}\int_{-\infty}^{\infty} \frac{\phi(t)}{x-t} dt
=\lim_{\varepsilon \to 0} \frac{1}{\pi}\int_{|t| \geq \varepsilon} \frac{\phi(t)}{x-t}dt  \quad (x \in \mathbb{R}),
$$
and extend it to a linear operator on $\widehat{B}_p(\mathbb{R})$ to be called the {\it Hilbert transform}. 
\end{definition}

There are several ways of extending $\mathcal{H}$ to a linear operator on $\widehat{B}_p(\mathbb{R})$. 
Especially on ${\rm BMO}(\mathbb R)$, we may use the regularized 
kernel
$$
\frac{1}{x-t}+\frac{1_{\{|t| \geq 1\}}}{t} \quad {\rm or} \quad \frac{1}{x-t}+\frac{t}{1+t^2}
$$
for $\mathcal{H}$ instead of $(x-t)^{-1}$. This defines the extension up to an additive constant. 
It is also possible to
extend $\mathcal{H}$ as the operator on certain distributions (see \cite[Chapter 4]{Pan}).

It is well-known that $\mathcal H$ gives a Banach automorphism of ${\rm BMO}(\mathbb R)$ 
satisfying ${\mathcal H} \circ {\mathcal H}=-I$
(see \cite[Chapter VI]{Ga} and \cite{Jia}). 
The Besov space $B_p^{\#}(\mathbb R)$ for $p \geq 1$ possesses the same property with $\mathcal H$ 
(see \cite[Proposition 4.7]{GP}).
Then, we obtain the following:

\begin{proposition}\label{Hilbert}
The Hilbert transform $\mathcal H$ is a Banach automorphism of $\widehat B_p(\mathbb R)$ 
satisfying ${\mathcal H} \circ {\mathcal H}=-I$ for $p \geq 1$.
\end{proposition}

Let
$P^\pm = \frac{1}{2}(I \pm i\mathcal{H})$, which we call the {\it Riesz projections}.
On the unit circle $\mathbb S$, the positive Riesz projection $P^+$ can be alternatively defined by the positive part of the Fourier transform of a function $\phi_*$ on $\mathbb S$, and the negative Riesz projection $P^-$ by the negative part.

We can apply the Riesz projections $P^\pm$ to $\widehat B_p(\mathbb R)$
by virtue of Proposition \ref{Hilbert}. We note that $P^++P^-=I$ and $P^+ \circ P^-=P^- \circ P^+=O$
by the definition and the property ${\mathcal H} \circ {\mathcal H}=-I$.
Moreover, 
the images of $P^\pm$ coincide with $E^\pm(\widehat {\mathcal B}_p(\mathbb H^\pm))$, which are the closed subspaces of 
$\widehat B_p(\mathbb R)$ consisting of all elements that extend to holomorphic functions on $\mathbb H^\pm$ by
the Poisson integral. 

\begin{proposition}\label{decomposition}
The Riesz projections $P^\pm$ in the Besov space $\widehat B_p(\mathbb R)$
for $p \geq 1$ are bounded linear projections onto the closed subspaces $E^\pm(\widehat {\mathcal B}_p(\mathbb H^+))$.
They yield the topological direct sum decomposition
$$
\widehat B_p(\mathbb R)=E^+(\widehat {\mathcal B}_p(\mathbb H^+)) \oplus E^-(\widehat {\mathcal B}_p(\mathbb H^-)).
$$
\end{proposition}

\begin{remark}
The boundedness of the Riesz projections $P^\pm$ on $B_p^{\#}(\mathbb S)$ for $p \geq 1$ was 
verified in \cite[p.449]{Pe} in a way of using the definition of $P^\pm$ by the Fourier
transform and the characterization of $B_p^{\#}(\mathbb S)$ in terms of approximation by polynomials.
From this, we see that $P^\pm$ is bounded on $\widehat B_p(\mathbb S)$,
and hence so is the Hilbert transform $\mathcal H=-i(2P^+-I)=i(2P^--I)$.
\end{remark}

\begin{definition}
Holomorphic functions of the upper and the lower half-planes $\mathbb H^\pm$ defined by
the Cauchy integrals of $\phi$,
$$
\frac{1}{\pi}\int_{-\infty}^{\infty} \frac{\phi(t)}{z-t} dt
\quad (z \in \mathbb{H}^\pm),
$$
are called the {\it Szeg\"o projections} of $\phi$.
\end{definition}

We see that the Szeg\"o projections
give the bounded linear maps $\widehat B_p(\mathbb R) \to \widehat {\mathcal B}_p(\mathbb H^\pm)$ whose composition
with the trace operators
$E^\pm:\widehat {\mathcal B}_p(\mathbb H^\pm) \to \widehat B_p(\mathbb R)$ coincide with the Riesz projections $P^\pm$
(by the Plemelj formula in the special case).
In the sequel, we do not distinguish them, denote both of them by $P^\pm$, and call the Szeg\"o projections.
Moreover, we regard $\widehat {\mathcal B}_p(\mathbb H^\pm)$
as the subspaces of $\widehat B_p(\mathbb R)$
by omitting $E^\pm$ and represent the direct sum decomposition of $\widehat B_p(\mathbb R)$ 
in Proposition \ref{decomposition} by 
$$
\widehat B_p(\mathbb R)=\widehat {\mathcal B}_p(\mathbb H^+) \oplus \widehat {\mathcal B}_p(\mathbb H^-).
$$

Finally in this section, we apply the above results to: 

\begin{proof}[Proof of Proposition \ref{isomorphism2}]
By the decomposition of $\widehat B_p(\mathbb R)$ as in Proposition \ref{decomposition},
it suffices to prove that $K_*:\widehat B_p(\mathbb R) \to
\widehat B_p(\mathbb S)$ is a Banach isomorphism restricted to
$\widehat {\mathcal B}_p(\mathbb H^\pm)$.
This has been verified by \cite[Theorem 2.4]{WM-*}, and thus we obtain the required result.
\end{proof}

\section{Composition Operators and interpolation}\label{5}

The composition operator $C_h$ defined for a quasisymmetric homeomorphism $h:\mathbb{R} \to \mathbb{R}$ 
is introduced as acting on 
the space of functions $\phi$ on $\mathbb{R}$ by $C_h(\phi)=\phi \circ h$. Then, we prove that
$C_h$ is a bounded linear operator on $\widehat B_p(\mathbb R)$. 

\begin{proposition}\label{p>1}
For $p>1$, the composition operator $C_h$ induced by any quasisymmetric homeomorphism $h$ gives
a Banach automorphism of ${B}_{p}(\mathbb{R})$ such that
the operator norm $\Vert C_h \Vert$ is bounded by a constant
depending only on the doubling (quasisymmetry) constant of $h$.
The converse also holds. 
\end{proposition}

The proof relies on interpolation of Banach spaces, as elucidated in \cite[Theorem 1.3]{BS} and \cite[Theorem 12]{B2}. 
For $p > 1$, the Besov space ${B}_{p}(\mathbb{R})$ can be regarded as the trace of a certain function space on $\mathbb{H}$, interpolated with the $2$-dimensional BMO function space ${\rm BMO}(\mathbb{H})$. The crucial role of a quasiconformal self-homeomorphism of $\mathbb{H}$ in this argument is that its composition operator
acts boundedly on ${\rm BMO}(\mathbb{H})$. This proof does not extend to $p=1$.

For $p=1$, we can obtain the following claim instead, 
which is enough for characterizing the quasisymmetric homeomorphisms in $T_1$.

\begin{theorem}\label{composition}
Suppose that a quasisymmetric homeomorphism $h$ is locally absolutely continuous and
satisfies $\log h' \in {\dot W}^1_1(\mathbb R)$. Then, the composition operator $C_h$
gives a Banach automorphism of $\widehat B_1(\mathbb R)$ such that
the operator norm $\Vert C_h \Vert$ is bounded by a constant
depending only on $ \Vert \log h' \Vert_{{\dot W}^1_{1}}$ and tending
$1$ as $ \Vert \log h' \Vert_{{\dot W}^1_{1}} \to 0$.
\end{theorem}

The proof of Theorem \ref{composition} utilizes the real interpolation of 
$B_1^{\#}(\mathbb{R})$ by ${\dot W}^2_1(\mathbb{R})$ and $L_1(\mathbb{R})$ as defined below
with its property given in the next proposition. See \cite[Chap.16]{Leo}
for interpolation of Banach spaces.
For each factor of this real interpolation of $B_1^{\#}(\mathbb{R})$, 
we will show that $C_h$ is a bounded linear operator 
to itself.
From this, we see that $C_h$ is also a bounded linear operator on $B_1^{\#}(\mathbb R)$.

\begin{definition}
Let $(X_0, \Vert \cdot \Vert_{X_0})$ and $(X_1, \Vert \cdot \Vert_{X_1})$ 
be Banach spaces that are continuously embedded into a common topological vector space,
which we call an admissible pair of Banach spaces. For $t>0$, 
$$
K(x,t)=\inf\,\{\Vert x_0 \Vert_{X_0}+t\Vert x_1 \Vert_{X_1} \mid x=x_0+x_1,\ x_0 \in X_0,\ x_1 \in X_1\} 
$$
is defined as the norm of $x \in X_0+X_1$. Then, for $\sigma \in (0,1)$ and $q \geq 1$, the set
$$
(X_0,X_1)_{\sigma,q}=
\{x \in X_0 +X_1 \mid \Vert x \Vert_{\sigma,q}=\left( \int_0^\infty K(x,t)^q 
\frac{dt}{t^{1+\sigma q}} \right)^{\frac{1}{q}}<\infty\}
$$ 
is defined to be the {\it real interpolation} of $X_0$ and $X_1$.
\end{definition}

The real interpolation space $(X_0,X_1)_{\sigma,q}$ is a Banach space with the norm $\Vert \cdot \Vert_{\sigma,q}$
(see \cite[Theorem 16.5]{Leo}). 
The boundedness of the associated operator with this interpolation is stated as follows (see \cite[Theorem 16.12]{Leo}).

\begin{proposition}\label{interpolation}
Let $(X_0, \Vert \cdot \Vert_{X_0})$ and $(X_1, \Vert \cdot \Vert_{X_1})$ 
be an admissible pair of Banach spaces. Let $C:X_0+X_1 \to X_0+X_1$ be a linear transformation
such that its restrictions $C|_{X_0}$ and $C|_{X_1}$ are bounded linear transformations
of $X_0$ and $X_1$, respectively. Then, for any $\sigma \in (0,1)$ and $q \geq 1$,
$C|_{(X_0,X_1)_{\sigma,q}}$ is a bounded linear transformation of
the real interpolation $(X_0,X_1)_{\sigma,q}$ and its operator norm satisfies
$$
\Vert C|_{(X_0,X_1)_{\sigma,q}} \Vert \leq \Vert C|_{X_0} \Vert^{1-\sigma} \Vert C|_{X_1} \Vert^\sigma.
$$
\end{proposition}

\begin{proof}[Proof of Theorem \ref{composition}]
Since $\widehat B_1(\mathbb R)=B_1^{\#}(\mathbb R) \cap {\dot W}^1_1(\mathbb R)$ by Lemma \ref{W1},
we first consider $B_1^{\#}(\mathbb R)$ by representing it as the real interpolation
$$
B_1^{\#}(\mathbb R)=({\dot W}^2_1(\mathbb R), L_1(\mathbb R))_{\frac{1}{2},1},
$$
where ${\dot W}^2_1(\mathbb R)$ is the homogeneous Sobolev space with norm
$\Vert \phi \Vert_{{\dot W}^2_1}=\Vert \phi'' \Vert_{L_1}$ (see \cite[Theorem 17.30]{Leo}).
Then, we take the intersection with ${\dot W}^1_1(\mathbb R)$ for recovering $\widehat B_1(\mathbb R)$.

The specificity of one-dimensional spaces is applied here: 
if $\log h' \in {\dot W}^1_{1}(\mathbb R)$, then $\log h'$ is bounded and continuous, and so are $h'$ and $(h')^{-1}$.
By the normalization $h(0)=0$ and $h(1)=1$, there is some $\xi \in [0,1]$ such that $h'(\xi)=1$.
This implies that $\Vert \log h' \Vert_{L_\infty} \leq \Vert \log h' \Vert_{{\dot W}^1_{1}}$.
From this, we have $1/M \leq h' (x) \leq M$ for some constant $M \geq 1$ depending only on $ \Vert \log h' \Vert_{{\dot W}^1_{1}}$ and tending $1$ as $ \Vert \log h' \Vert_{{\dot W}^1_{1}} \to 0$.
Moreover, by $(\log h')'=h''/h'$, we have 
$h \in {\dot W}^2_{1}(\mathbb R)$ with
$\Vert h \Vert_{{\dot W}^2_{1}} \leq M \Vert \log h' \Vert_{{\dot W}^1_{1}}$.

From these conditions, we see that
$C_h$ is bounded linear operators of $L_1(\mathbb R)$,
${\dot W}^1_1(\mathbb R)$, and ${\dot W}^2_1(\mathbb R)$ onto themselves:
\begin{align}
\int_{\mathbb R} |\phi \circ h(x)|dx&=\int_{\mathbb R} |\phi(h(x))|(h'(x))^{-1}h'(x)dx \leq M \int_{\mathbb R} |\phi (t)|dt;\\
\int_{\mathbb R} |(\phi \circ h)'(x)|dx&=\int_{\mathbb R} |\phi' (h(x))|h'(x)dx=\int_{\mathbb R} |\phi' (t)|dt;\\
\end{align}
\vspace{-1cm}
\begin{align}
\int_{\mathbb R} |(\phi \circ h)''(x)|dx&\leq\int_{\mathbb R} |\phi'' (h(x))|h'(x)^2dx
+\int_{\mathbb R} |\phi' (h(x))||h''(x)|dx\\
&\leq M \int_{\mathbb R} |\phi'' (t)|dt+M \Vert \log h' \Vert_{{\dot W}^1_{1}}
\Vert \phi' \Vert_{L_\infty}\\
&\leq M (1+ \Vert \log h' \Vert_{{\dot W}^1_{1}})\int_{\mathbb R} |\phi'' (t)|dt.
\end{align}
Here, in the last inequality above, we have used the fact that if $\phi \in {\dot W}^2_{1}(\mathbb R)$ then
$\phi' \in L_\infty(\mathbb R)$ and $\Vert \phi' \Vert_{L_\infty} \leq \Vert \phi \Vert_{{\dot W}^2_{1}}=\int_{\mathbb R} |\phi'' (t)|dt$
under assumption $\phi'(0)=0$.

In the above argument, we also have shown that the operator norm of 
the composition operator $C_h$ on $L_1(\mathbb R)$,
${\dot W}^1_1(\mathbb R)$, and ${\dot W}^2_1(\mathbb R)$ are all bounded 
in terms of $\Vert \log h' \Vert_{{\dot W}^1_{1}}$.
Therefore, by Proposition \ref{interpolation}, we obtain the required statement.
\end{proof}

\section{Conformal welding}\label{6}

In general, a quasisymmetric homeomorphism $h:\mathbb R \to \mathbb R$ can be expressed as the discrepancy between the boundary values 
$f_1=F_1|_{\mathbb R}$ and $f_2=F_2|_{\mathbb R}$ of two conformal homeomorphisms $F_1:\mathbb H^- \to \Omega^-$ and 
$F_2:\mathbb H^+ \to \Omega^+$, where $\Omega^-$ and $\Omega^+$ are complementary domains in $\mathbb C$ with
$\partial \Omega^-=\partial \Omega^+$. This expression $h=f_2^{-1} \circ f_1$ is called {\it conformal welding}.
If absolute continuity is present, then
$\log f_2' \circ h + \log h' = \log f_1'$
is satisfied.
This allows us to see that the property of $\log h'$ is determined by that of $\log f_1'$ and $\log f_2'$
if the composition operator $C_h$ preserves it.

We proceed to the problem of finding the class of quasisymmetric homeomorphisms $h$
whose quasiconformal extensions $H(\mu)$ to $\mathbb H$ have complex dilatations 
$\mu \in M_p(\mathbb H)$. First, we consider the boundary extension  
of a conformal homeomorphism of $\mathbb H$.

\begin{lemma}\label{app}
For $\mu \in M_p(\mathbb H^+)$, let $F^\mu$ be the normalized quasiconformal self-homeo\-morphism of $\mathbb C$
that is conformal on $\mathbb H^-$ and has the complex dilatation $\mu$ on $\mathbb H^+$. Then,
$f=F^\mu|_{\mathbb R}$ is locally absolutely continuous, $f'(x) \neq 0$ almost everywhere on $\mathbb R$, and
$\log f'$ belongs to $\widehat{B}_p(\mathbb R)$.
\end{lemma}

\begin{proof}
By Lemma \ref{conformal}, we see that $\Phi=\log (F^\mu|_{\mathbb H^-})'$ belongs to 
$\widehat{\mathcal B}_p(\mathbb H^-)$. Moreover, by Proposition \ref{trace},
the boundary extension $\phi=E(\Phi)$ is in $\widehat{B}_p(\mathbb R)$, and in particular,
$\phi(x)$ is finite almost everywhere on $\mathbb R$. Since $f(\mathbb R)$ is a chord-arc curve,
it is well-known that $f$ is locally absolutely continuous and $f'(x)$ coincides with the angular
derivative of $F^\mu$ at $x$ almost everywhere on $\mathbb R$ (see \cite[Lemma 3.2]{WM-4}), which is non-zero
almost everywhere. 
Hence, $\log f'=\phi$. This completes the proof.
\end{proof}

Representing a quasisymmetric homeomorphism $h=H(\mu)|_{\mathbb R}$ with $\mu \in M_p(\mathbb H)$
by conformal welding, we obtain the
following characterization of it in one direction. The other direction will be shown later.
The boundedness of the composition operator $C_h$ is crucial in this proof.

\begin{theorem}\label{main1}
Let $h:\mathbb R \to \mathbb R$ be a quasisymmetric homeomorphism whose 
quasiconformal extension $H=H(\mu):\mathbb H \to \mathbb H$ has a complex dilatation $\mu$
belonging to $M_p(\mathbb H)$ for $p \geq 1$. Then, $h$ is locally absolutely continuous and 
$\phi=\log h'$ belongs to $\widehat B_p(\mathbb R)$.
\end{theorem}

\begin{proof}
We represent $h$ by conformal welding. Let $F_1=F^\mu:\mathbb C \to \mathbb C$ be the normalized quasiconformal homeomorphism
that is conformal on $\mathbb H^-$ and has the complex dilatation $\mu$ on $\mathbb H^+$.
Let $F_2=F_{\bar \mu^{-1}}:\mathbb C \to \mathbb C$ be the normalized quasiconformal homeomorphism
that is conformal on $\mathbb H^+$ and has the complex dilatation $\bar \mu^{-1}$ on $\mathbb H^-$.
Here, $\mu^{-1}$ is the complex dilatation of the quasiconformal homeomorphism $H({\mu^{-1}}):\mathbb H^+ \to \mathbb H^+$
satisfying $(H(\mu))^{-1}=H(\mu^{-1})$, and $\bar \mu^{-1}$ stands for its reflection defined on $\mathbb H^-$.
By choosing a suitable representative $\mu$ in the Teich\-m\"ul\-ler equivalence class $[\mu]$, 
we see that $\mu^{-1}$ also belongs to $M_p(\mathbb H)$ (see \cite[Lemma 3.1]{WM-*}).
We note that the normalization of $F_1$ and $F_2$ guarantees $F_1(\mathbb R)=F_2(\mathbb R)$.

Let $f_1=F_1|_{\mathbb R}$ and $f_2=F_2|_{\mathbb R}$. Then, we have that $f_2 \circ h=f_1$. 
Here, Lemma \ref{app} asserts that $f_1$ and $f_2$ are locally absolutely continuous, 
and $\log f_1'$ and $\log f_2'$ belong to $\widehat{B}_p(\mathbb R)$. Moreover,
$f_2'(x) \neq 0\ {\rm (a.e.)}$ implies that $h$ is also locally absolutely continuous.
Then, taking the logarithm and the derivatives for the conformal welding $f_2 \circ h=f_1$, we have
\begin{equation}\label{welding}
\log f_2' \circ h+\log h'=\log f_1'.
\end{equation}

In the case $p=1$, 
we see that $\log f_2' \circ h \in \dot W_1^1(\mathbb R)$ because
$$
\int_{\mathbb R}|(\log f_2' \circ h)'(x)|dx=\int_{\mathbb R}|(\log f_2')'(h(x))|h'(x)dx=\int_{\mathbb R}|(\log f_2')'(t)|dt
$$
is finite by $\log f_2' \in \widehat B_1(\mathbb R) \subset \dot W_1^1(\mathbb R)$.
Since $\log f_1'$ also belongs to $\widehat B_1(\mathbb R) \subset {\dot W}^1_{1}(\mathbb R)$, we have
$\log h' \in {\dot W}^1_{1}(\mathbb R)$ by \eqref{welding}.
This is the assumption of Theorem \ref{composition}, so we can apply it to conclude that the composition operator
$C_h$ is bounded on $\widehat B_1(\mathbb R)$. This implies that $C_h(\log f_2')=\log f_2' \circ h \in \widehat B_1(\mathbb R)$.
Hence, we obtain $\log h' \in \widehat B_1(\mathbb R)$ again by \eqref{welding}.

In the case of $p>1$, Proposition \ref{p>1} claims that the composition operator $C_h$ is bounded on 
$\widehat B_p(\mathbb R)$. It follows from \eqref{welding} that $\log h' \in \widehat B_p(\mathbb R)$ 
in the same manner as above.
\end{proof}

For the case $p=1$, Theorem \ref{main1} in particular implies that $\log h'$ belongs to
$$
\widehat B_1(\mathbb R) \subset {\dot W}_1^1(\mathbb R) \subset C^0(\mathbb R) \cap L_\infty(\mathbb R).
$$
From this, the following result follows.

\begin{corollary}
Every quasisymmetric homeomorphism $h$ in $T_1$ is a $C^1$-diffeomorphism of $\widehat{\mathbb R} \cong \mathbb S$ 
(including $\infty$).
\end{corollary}

This result was proved in \cite{AB}. 
Its proof utilizes the {\it Teich\-m\"ul\-ler--Wittich--Belinski\v{\i} theorem} (see \cite{Shishi}), 
which asserts that for a quasiconformal map $H$ with complex dilatation $\mu$, 
the local integrability condition at a point $z_0$ given by
$$
\int_{|z-z_0|<r} \frac{|\mu(z)|}{|z-z_0|^2}dxdy<\infty
$$
for some $r>0$ implies that $H$ has a non-zero complex derivative $H'(z_0)$ at $z_0$. If the complex dilatation $\mu$ of 
a quasiconformal map $H:\mathbb H^+ \to \mathbb H^+$ belongs to $M_1(\mathbb H^+)$, then the extension of $H$ 
to a quasiconformal homeomorphism of $\mathbb C$ by taking the reflection with respect to $\mathbb R$ (with complex dilatation $\bar \mu$ on $\mathbb H^-$) satisfies the above condition at any $z_0 \in \mathbb R$.

\section{Weil--Petersson Embedding in Bers Coordinates}\label{7}

We regard Weil--Petersson curves as the images of certain embeddings $\gamma:\mathbb{R} \to \mathbb C$, and we use Teich\-m\"ul\-ler space to coordinate these embeddings. 
We first recall the integrable Teich\-m\"ul\-ler space $T_p$ defined in Section \ref{2}.

For $p \geq 1$, the $p$-integrable Teich\-m\"ul\-ler space $T_p$ is the
set of all normalized quasisymmetric homeomorphisms $h:\mathbb R \to \mathbb R$
that can be extended to quasiconformal homeomorphisms $H:\mathbb H \to \mathbb H$
whose complex dilatations belong to $M_p(\mathbb H)$. By this correspondence from Beltrami coefficients $\mu$
to quasisymmetric homeomorphisms $h(\mu)$ via quasiconformal ones $H(\mu)$, we have
the Teich\-m\"ul\-ler projection $\pi: M_p(\mathbb H) \to T_p$.
An element $h(\mu)$ of $T_p$ can be represented by a Teich\-m\"ul\-ler equivalence class $[\mu]$ for $\mu \in M_p(\mathbb H)$.

Let $F^\mu$ denote the normalized quasiconformal homeomorphism of $\mathbb C$ whose complex dilatation
is $\mu \in M_p(\mathbb H^+)$ on $\mathbb H^+$ and $0$ on $\mathbb H^-$.
The Schwarzian derivative map $S:M_p(\mathbb H^+) \to {\mathcal A}_p(\mathbb H^-)$ is defined by the
correspondence $\mu \mapsto S(F^\mu|_{\mathbb H^-})$.
This induces
the Bers embedding $\alpha:T_p \to {\mathcal A}_p(\mathbb H^-)$ such that
$\alpha \circ \pi=S$. By showing that $\alpha$ is a topological embedding, we provide the
complex Banach structure for $T_p$.

In a similar way,
the pre-Schwarzian derivative map $L:M_p(\mathbb H^+) \to \widehat{\mathcal B}_p(\mathbb H^-)$
is defined by the
correspondence $\mu \mapsto \log (F^\mu|_{\mathbb H^-})'$. This induces
a well-defined injection $\beta:T_p \to \widehat{\mathcal B}_p(\mathbb H^-)$ such that
$\beta \circ \pi=L$. We call $\beta$ the {\it pre-Bers embedding}.
The following claim is known to be true for all $p \geq 1$ in \cite[Theorem 5.1]{WM-*}.

\begin{proposition}\label{pre-Bers}
The pre-Bers embedding $\beta:T_p \to \widehat{\mathcal B}_p(\mathbb H^-)$ for $p \geq 1$ is
a biholomorphic homeomorphism onto the connected open subset $\beta(T_p)=L(M_p(\mathbb H^+))$ in $\widehat{\mathcal B}_p(\mathbb H^-)$.
\end{proposition}

The integrable Teich\-m\"ul\-ler space $T_p$ possesses the group structure as a subgroup of $T$.
For every $[\nu] \in T_p$, the right translation $r_{[\nu]}:T_p \to T_p$ of the group elements of $T_p$
is defined by $[\mu] \mapsto [\mu] \ast [\nu]$.
The following results are proved in \cite[Theorem 6.1]{WM-4} and \cite[Section 4]{WM-1}.

\begin{proposition}\label{group}
For $p \geq 1$, $T_p$ is a topological group. Moreover,
every right translation $r_{[\nu]}$ is a biholomorphic automorphism of $T_p$.
\end{proposition}

We introduce the Weil--Petersson embeddings as follows. 

\begin{definition}
For $\mu^+ \in M_p(\mathbb H^+)$ and $\mu^- \in M_p(\mathbb H^-)$, we denote by $G(\mu^+,\mu^-)$
the normalized quasiconformal self-homeomorphism of $\mathbb C$ whose complex dilatation on $\mathbb H^+$
is $\mu^+$ and on $\mathbb H^-$ is $\mu^-$. The restriction of $G(\mu^+,\mu^-)$ to $\mathbb R$ is called
a {\it $p$-Weil--Petersson embedding}. Its image is referred to as a $p$-Weil--Petersson curve. 
\end{definition}

The following claim guarantees that
$p$-Weil--Petersson embeddings are parametrized by the product of Teich\-m\"ul\-ler spaces.
The proof is the same as that in \cite[Proposition 4.1]{WM-2}.

\begin{proposition}
$G(\mu^+,\mu^-)|_{\mathbb R}=G(\nu^+,\nu^-)|_{\mathbb R}$ if and only if $[\mu^+]=[\nu^+]$ and
$[\mu^-]=[\nu^-]$.
\end{proposition}

Hence, a $p$-Weil--Petersson embedding $G(\mu^+,\mu^-)|_{\mathbb R}$ is determined by
the pair of the Teich\-m\"ul\-ler equivalence classes, 
which is denoted by $\gamma([\mu^+],[\mu^-])$.
Let $T_p^+=\pi(M_p(\mathbb H^+))$ and $T_p^-=\pi(M_p(\mathbb H^-))$.
Then, $T_p^+ \times T_p^-$ becomes the coordinates of the space of $p$-Weil--Petersson embeddings.
For the space of quasi-Fuchsian groups, this method is called
{\it simultaneous uniformization}, which is originally due to Bers. 
The entire space of quasi-Fuchsian groups can be coordinated as a product of Teich\-m\"ul\-ler spaces. 
Similarly, the entire space of $p$-Weil--Petersson embeddings can be coordinated as $T_p^+ \times T_p^-$. 
We refer to this as {\it the Bers coordinates}. 

Let $\bar \mu \in M_p(\mathbb H^-)$ denote the reflection $\overline{\mu(\bar z)}$ 
of a Beltrami coefficient $\mu(z)$ for $z \in \mathbb H^+$.
Then, $G(\mu, \bar \mu)$ is nothing but the normalized quasiconformal homeomorphism $H(\mu):\mathbb C \to \mathbb C$ preserving $\mathbb R$,
and $h=G(\mu, \bar \mu)|_{\mathbb R}=\gamma([\mu],[\bar \mu])$ is the corresponding quasisymmetric homeomorphism of $\mathbb R$,
which can be regarded as an element of the Teich\-m\"ul\-ler space $T_p$. 
The {\it axis of symmetry} of the product space $T_p^+ \times T_p^-$ is defined as
$$
{\rm Sym}\,(T_p^+ \times T_p^-)=\{([\mu],[\bar \mu]) \mid [\mu] \in T_p\}.
$$
The canonical map $\iota:T_p \to {\rm Sym}\,(T_p^+ \times T_p^-) \subset T_p^+ \times T_p^-$
defined by $[\mu] \mapsto ([\mu],[\bar \mu])$ is a real-analytic embedding, and hence
${\rm Sym}\,(T_p^+ \times T_p^-)$ is a real-analytic submanifold of $T_p^+ \times T_p^-$.

Regarding the group structure of $T_p$, the right transformation $r_{[\nu]}$ defined by $[\nu] \in T_p$ is extended to $T_p^+ \times T_p^-$, defining the parallel translation 
$$
R_{[\nu]}([\mu^+],[\mu^-]) = (r_{[\nu]}([\mu^+]), r_{[\bar \nu]}([\mu^-]))=([\mu^+] \ast [\nu], [\mu^-] \ast [\bar \nu]).
$$
This is a biholomorphic automorphism of $T_p^+ \times T_p^-$ that preserves  
${\rm Sym}\,(T_p^+ \times T_p^-)$.

Theorem \ref{main1} shows that
any element $h=\gamma([\mu],[\bar \mu]) \in T_p$ with $\mu \in M_p(\mathbb H^+)$ for $p \geq 1$ is
locally absolutely continuous and $\log h'$ belongs to ${\rm Re}\,\widehat B_p(\mathbb R)$.
This claim can be generalized as follows using the method of conformal welding. 

\begin{corollary}\label{main-b}
Let $\gamma=\gamma([\mu^+],[\mu^-])$ be a $p$-Weil--Petersson embedding for $p \geq 1$.
Then, $\gamma:\mathbb R \to \mathbb C$ is locally absolutely continuous and
$\log \gamma'$ belongs to $\widehat{B}_p(\mathbb R)$.
\end{corollary}

\begin{proof}
We represent $\gamma$ as the following composition:
$$
\gamma([\mu^+],[\mu^-])=\gamma([0],[\mu^-]\ast[\overline{\mu^+}]^{-1}) \circ \gamma([\mu^+],[\overline{\mu^+}]).
$$
Here, $\gamma_1=\gamma([0],[\mu^-]\ast[\overline{\mu^+}]^{-1})=F_{\nu}|_{\mathbb R}$
for the quasiconformal self-homeomorphism $F_{\nu}$ of $\mathbb C$ that is conformal 
on $\mathbb H^+$ and has complex dilatation $\nu$ on $\mathbb H^-$ with $[\nu]=[\mu^-]\ast[\overline{\mu^+}]^{-1}$, and
$\gamma_2=\gamma([\mu^+],[\overline{\mu^+}])=H(\mu^+)|_{\mathbb R}$ for the quasiconformal self-homeomorphism $H(\mu^+)$
of $\mathbb C$ with the indicated complex dilatation. Then, $\gamma_1$
is locally absolutely continuous and $\log \gamma_1' \in \widehat{B}_p(\mathbb R)$ by
Lemma \ref{app}. In addition, $\gamma_2$
is locally absolutely continuous and $\log \gamma_2' \in \widehat{B}_p(\mathbb R)$ by Theorem \ref{main1}.
Therefore, $\gamma=\gamma_1 \circ \gamma_2$ is locally absolutely continuous and
$\log \gamma'=\log \gamma_1' \circ \gamma_2+\log \gamma_2'$ belongs to $\widehat{B}_p(\mathbb R)$ by
Proposition \ref{p>1} and Theorem \ref{composition}.
\end{proof}

\section{Holomorphy to the Besov space and the characterization of $T_p$}\label{8}

For a Weil--Petersson embedding $\gamma:\mathbb R \to \mathbb C$,
$\log \gamma'$ belongs to the Besov space by Corollary \ref{main-b}.
In this section, we first prove that this correspondence is holomorphic in the Bers coordinates.
Then, using this fact, we extend Theorem \ref{main1} to the complete characterization of
$T_p$ in terms of the real Besov space ${\rm Re}\,\widehat{B}_p(\mathbb R)$ (Theorem \ref{theorem10}).

\begin{lemma}\label{holomorphic}
The map $\Lambda:T_p^+ \times T_p^- \to \widehat{B}_p(\mathbb R)$ defined by
$([\mu^+],[\mu^-]) \mapsto \log \gamma'$ for $\gamma=\gamma([\mu^+],[\mu^-])$
is a holomorphic injection for $p \geq 1$.
\end{lemma}

\begin{proof}
By the Hartogs theorem for Banach spaces (see \cite[\S 14.27]{Ch}), to see that $\Lambda$ is holomorphic
it suffices to show that $\Lambda$ is separately holomorphic.
Namely, we fix, say $[\mu^+_0] \in T_p^+$, and prove that $\Lambda([\mu^+_0],[\mu^-])$ is holomorphic in 
$[\mu^-] \in T_p^-$. The other case is treated in the same way.

Let $h_0=\pi(\mu^+_0) \in T_p$ be the quasisymmetric homeomorphism of $\mathbb R$, and $C_{h_0}$ 
the composition operator on $\widehat{B}_p(\mathbb R)$ induced by $h_0$. We define
the affine translation $Q_{h_0}$ on $\widehat{B}_p(\mathbb R)$ by $Q_{h_0}(\phi)=C_{h_0}(\phi)+\log h_0'$.
Then, 
\begin{equation}\label{translation}
\Lambda \circ R_{[\mu^+_0]} = Q_{h_0} \circ \Lambda 
\end{equation}
holds (see \cite[Proposition 4.1]{WM-4}), and this relation yields
a useful representation
\begin{equation}
\Lambda([\mu^+_0],\,\cdot \,)=Q_{h_0} \circ \Lambda([0],r_{[\,\bar{\mu}^+_0\,]}^{-1}(\,\cdot \,)) .
\end{equation}
Here, 
$\Lambda([0],\,\cdot \,)$ is regarded as the trace operator in Proposition \ref{trace} if
we compose the pre-Bers embedding $\beta^-:T_p^- \to \widehat{\mathcal B}_p(\mathbb H^+)$, that is,
$$
\Lambda([0],[\mu])=E^+(\beta^-([\mu])) \quad ([\mu] \in T_p^-).
$$
Since $E^+$ is a bounded linear operator and $r_{[\,\bar{\mu}^+_0\,]}^{-1}$ is holomorphic, we conclude that
$\Lambda([\mu^+_0],\,\cdot \,)$ is holomorphic.

Suppose that $\Lambda([\mu_1^+],[\mu_1^-])=\Lambda([\mu_2^+],[\mu_2^-])$. Then,
$\gamma([\mu_1^+],[\mu_1^-])=\gamma([\mu_2^+],[\mu_2^-])$ by the normalization.
This implies that $[\mu_1^+]=[\mu_2^+]$ and $[\mu_1^-]=[\mu_2^-]$, which can be verified
by the same proof as that of \cite[Proposition 4.1]{WM-2}.
Hence, $\Lambda$ is injective.
\end{proof}

We utilize a certain smaller subspace of $T_p$ to prove the converse of
Theorem \ref{main1}. Let $M_0(\mathbb H)$ denote the set of all Beltrami coefficients on
$\mathbb H$ with compact support, and define $T_0=\pi(M_0(\mathbb H))$. Then, 
$T_0$ is a dense subset of $T_p$ for every $p \geq 1$. This is because
for every $[\mu] \in T_p$ there exists a representative $\nu \in M_p(\mathbb H)$ of $[\mu]$ such that
the cut-off sequence $\nu 1_{W_n} \in M_0(\mathbb H)$ 
for an exhaustion of $\mathbb H$
by compact subsets $W_n$ converges to $\nu$ in $M_p(\mathbb H)$ in the norm 
$\Vert \cdot \Vert_p+\Vert \cdot \Vert_\infty$. 
The appropriate representative $\nu$ can be constructed by a finite composition of
quasiconformal self-homeomorphisms of $\mathbb H$ given by the Ahlfors--Weill section
as in \cite[Lemma 3.4]{WM-1}.

A quasisymmetric homeomorphism $h:\mathbb R \to \mathbb R$ in the universal Teich\-m\"ul\-ler space $T$ belongs to $T_0$ if and only if
$h$ is a real-analytic self-diffeomorphism of $\widehat {\mathbb R}=\mathbb R \cup \{\infty\}$, where
the regularity at $\infty$ is given by that of $h(1/x)$ at $x=0$. Moreover, we see that a holomorphic function
$\Phi \in \beta(T)$ on $\mathbb H$ belongs to $\beta(T_0)$ if and only if $\Phi$ extends to
$\widehat {\mathbb R}$ analytically.
We can easily generalize this property as follows.

\begin{proposition}\label{real-analytic}
A pair
$([\mu^+],[\mu^-]) \in T^+ \times T^-$ belongs to $T_0^+ \times T_0^-$ if and only if $\gamma=\gamma([\mu^+],[\mu^-])$ is a
real-analytic diffeomorphism of $\widehat {\mathbb R}$ into the Riemann sphere $\widehat {\mathbb C}$.
In this case, $\phi=\log \gamma'$ is a real-analytic function on $\widehat{\mathbb R}$. 
\end{proposition}

Let $B_0(\mathbb R)$ be the complex linear subspace of ${\rm BMO}(\mathbb R)$ consisting of all
complex-valued real-analytic functions on $\widehat {\mathbb R}$. Then, $B_0(\mathbb R) \subset \widehat{B}_p(\mathbb R)$
for every $p \geq 1$. 
Since $T_0$ is dense in $T_p$, $T_0^+ \times T_0^-$ is also dense in $T_p^+ \times T_p^-$.
Then, by Lemma \ref{holomorphic} and Proposition \ref{real-analytic}, if we have that 
$\Lambda:T_p^+ \times T_p^- \to \widehat{B}_p(\mathbb R)$ is surjective near the origin,
we see that $B_0(\mathbb R)$ is a dense subspace of $\widehat{B}_p(\mathbb R)$.

We denote the tangent space of $T_p$ at $[\mu]$ by ${\mathscr T}_{[\mu]}T_p$.
The tangent space of $T_p^+ \times T_p^-$ at $([\mu^+],[\mu^-])$ is represented by the direct sum
\begin{equation}\label{tangentspace}
{\mathscr T}_{([\mu^+],[\mu^-])}(T_p^+ \times T_p^-)={\mathscr T}_{[\mu^+]}T_p^+ \oplus {\mathscr T}_{[\mu^-]}T_p^-.
\end{equation}
By the identification $T_p^+ \cong \beta(T_p^+) \subset \widehat{\mathcal B}_p(\mathbb H^-)$
and $T_p^- \cong \beta(T_p^-) \subset \widehat{\mathcal B}_p(\mathbb H^+)$ under the pre-Bers embedding by Proposition \ref{pre-Bers},
we may assume that ${\mathscr T}_{[\mu^+]}T_p^+ \cong \widehat{\mathcal B}_p(\mathbb H^-)$ and
${\mathscr T}_{[\mu^-]}T_p^- \cong \widehat{\mathcal B}_p(\mathbb H^+)$.
Then, 
the derivative $d_{([\mu^+],[\mu^-])} \Lambda$ of $\Lambda$ at $([\mu^+],[\mu^-])$ is regarded as
the linear mapping
$$
d_{([\mu^+],[\mu^-])}\, \Lambda:\widehat{\mathcal B}_p(\mathbb H^-) \oplus \widehat{\mathcal B}_p(\mathbb H^+)
\to \widehat B_p(\mathbb R)
=\widehat {\mathcal B}_p(\mathbb H^-) \oplus \widehat {\mathcal B}_p(\mathbb H^+)
$$
taking the direct product decomposition as in Proposition \ref{decomposition} into account.

The derivative $d_{([0],[0])}\,\Lambda$ at the origin can be easily understood.
By checking that the restriction of $\Lambda$ to $T_p^+$ and $T_p^-$ coincides with   
the pre-Bers embeddings
$$
\Lambda|_{T_p^+ \times \{[0]\}}=\beta^+:T_p^+ \to \widehat {\mathcal B}_p(\mathbb H^-),\quad
\Lambda|_{\{[0]\} \times T_p^-}=\beta^-:T_p^- \to \widehat {\mathcal B}_p(\mathbb H^+),
$$
we see that the linearization $d_{([0],[0])}\, \Lambda$ is the identity map of $\widehat{\mathcal B}_p(\mathbb H^-) \oplus \widehat{\mathcal B}_p(\mathbb H^+)$.
This implies the following claim by the inverse mapping theorem (see \cite[\S 7.18]{Ch}). 

\begin{proposition}\label{derivative0}
The derivative $d_{([0],[0])}\, \Lambda$ is surjective, and hence $\Lambda^{-1}$ is holomorphic in
some neighborhood $U$ of $0$ in 
$\widehat{B}_p(\mathbb R)$.
\end{proposition}

Since $U$ is in the image of $\Lambda$,
we also obtain that $B_0(\mathbb R)$ is dense in $\widehat{B}_p(\mathbb R)$
as mentioned after Proposition \ref{real-analytic}.

Theorem \ref{main1} implies that $\Lambda([\mu], [\bar \mu]) \in {\rm Re}\, \widehat{B}_p(\mathbb R)$ for every
$([\mu], [\bar \mu]) \in {\rm Sym}\,(T_p^+ \times T_p^-)$. The converse of this claim also holds.

\begin{lemma}\label{converse}
For every $\phi \in {\rm Re}\, \widehat{B}_p(\mathbb R)$, there exists $([\mu], [\bar \mu]) \in {\rm Sym}\,(T_p^+ \times T_p^-)$
such that $\Lambda([\mu], [\bar \mu])=\phi$.
\end{lemma}

\begin{proof}
We take any $\phi_0$ from ${\rm Re}\,B_0(\mathbb R)$, which is dense in ${\rm Re}\, \widehat{B}_p(\mathbb R)$.
Then, 
$\gamma_0(x)=\int_0^x \exp{\phi_0(t)}dt$ is a real-analytic self-diffeomorphism of $\widehat{\mathbb R}$, and
by Proposition \ref{real-analytic}, we can find $([\mu_0], [\bar \mu_0]) \in{\rm Sym}\,(T_0^+ \times T_0^-)$
such that $\Lambda([\mu_0], [\bar \mu_0])=\phi_0$. By proposition \ref{derivative0}, there exists 
a neighborhood $U$ of $0 \in \widehat{B}_p(\mathbb R)$ that
is contained in the image of $\Lambda$.
Then, by relation \eqref{translation}, $R_{[\mu_0]} \circ \Lambda^{-1}|_U \circ Q_{h_0}^{-1}$ for $h_0=\pi(\mu_0)$ gives
the local holomorphic inverse of $\Lambda$ on the neighborhood $Q_{h_0}(U)$ of $\phi_0$. In particular, 
we see that $\Lambda^{-1}(\phi)$ belongs to ${\rm Sym}\,(T_p^+ \times T_p^-)$ for every 
$\phi \in {\rm Re}\, \widehat{B}_p(\mathbb R) \cap Q_{h_0}(U)$.
\end{proof}

\begin{remark}
For $p > 1$, the proof can also be demonstrated by applying the variant of the Beurling--Ahlfors 
quasiconformal extension,
which replaces the convolution kernel in the original definition with a Gaussian function (heat kernel) (see \cite{FKP}, \cite{WM-0}), as shown in \cite{WM-3}. In fact, this yields a
holomorphic right inverse of $\Lambda \circ(\pi^+ \times \pi^-)$ from 
a neighborhood of ${\rm Re}\, {B}_p(\mathbb R)$ in ${B}_p(\mathbb R)$ to $M_p(\mathbb H^+) \times M_p(\mathbb H^-)$.
\end{remark}

Therefore, combined with Lemma \ref{converse},
Theorem \ref{main1} is improved to the complete characterization of $T_p$ 
in terms of the real Besov space ${\rm Re}\,\widehat B_p(\mathbb R)$ for $p \geq 1$.
Thus, we obtain Theorem \ref{theorem10}.

\section{Biholomorphic correspondence}\label{9}

We have seen that 
$p$-Weil--Petersson embeddings $\gamma$ are represented in the Bers coordinates
and the map $\Lambda$ to the Besov space via $\log \gamma'$ is a holomorphic injection.
In this section, we prove that this is in fact a biholomorphic homeomorphism onto its image.

\begin{theorem}\label{main2}
Let $\Lambda:T_p^+ \times T_p^- \to \widehat B_p(\mathbb R)$ be the holomorphic injection
given by
$$
\Lambda([\mu^+],[\mu^-])= \log \gamma' \quad (\gamma=\gamma([\mu^+],[\mu^-]))
$$
for $p \geq 1$. Then,
its image ${\rm Ran}\,\Lambda$ is a connected open subset of $\widehat B_p(\mathbb R)$ containing
${\rm Re}\,\widehat B_p(\mathbb R)$,
and $\Lambda$ is a biholomorphic homeomorphism onto ${\rm Ran}\,\Lambda$.
\end{theorem}

In virtue of the fact that $\Lambda$ is a holomorphic injection,
to see that the inverse $\Lambda^{-1}$ is holomorphic,
we have only to show that its derivative $d\Lambda$ is surjective by the inverse mapping theorem.
The proof for the surjectivity of the derivative 
$$
d_{([\mu^+],[\mu^-])}\Lambda: {\mathscr T}_{([\mu^+],[\mu^-])}(T_p^+ \times T_p^-) \to \widehat B_p(\mathbb R)
$$
at $([\mu^+],[\mu^-]) \in T_p^+ \times T_p^-$
involves the following steps:
\begin{itemize}
\item[(a)] 
The image of the tangent space
$
{\mathscr T}_{([\mu^+],[\mu^-])}(T_p^+ \times T_p^-) \cong \widehat{\mathcal B}_p(\mathbb H^-) \oplus \widehat{\mathcal B}_p(\mathbb H^+)
$
under the derivative $d\Lambda$ is given as the algebraic sum
$$
{\rm Ran}\,d_{([\mu^+],[\mu^-])}\Lambda = C_{h^-}\widehat{\mathcal B}_p(\mathbb H^-) + C_{h^+}\widehat{\mathcal B}_p(\mathbb H^+),
$$
where $h^+=\pi(\mu^+)$ and $h^-=\pi(\mu^-)$.
\item[(b)] 
If $\Lambda([\mu^+],[\mu^-])$ is
on the subspace $Z_p=i\,{\rm Re}\,\widehat B_p(\mathbb R) \cap {\rm Ran}\, \Lambda$, then 
$d\Lambda$ is surjective on $\Lambda^{-1}(Z_p)$.
\item[(c)] 
For $\Lambda([\mu^+],[\mu^-]) = \log \gamma' \in Z_p$, we translate $([\mu^+],[\mu^-])$ in parallel by $R_{[\nu]}$
and move $\log \gamma'$ through the affine translation $Q_{h}$ with $\pi(\nu)=h$. 
Then, by $\Lambda \circ R_{[\nu]} = Q_{h} \circ \Lambda$ 
in \eqref{translation}, 
the surjectivity of $d\Lambda$ at any point of $T_p^+ \times T_p^-$ is demonstrated.
\end{itemize}

We show these steps
by partially relying on the previous result in \cite{WM-2} regarding the biholomorphy of $\Lambda$ on a certain larger space.
Let $T_B$ be the BMO Teich\-m\"ul\-ler space consisting of all normalized strongly quasisymmetric homeomorphisms 
$h:\mathbb R \to \mathbb R$ such that $\log h' \in {\rm BMO}(\mathbb R)$.
Likewise to the case of $T_p$, we can define the corresponding map 
$\widetilde \Lambda:T_B^+ \times T_B^- \to {\rm BMO}(\mathbb R)$. 
The embedding $\gamma:\mathbb R \to \mathbb C$ similarly defined by $\gamma=\gamma([\mu^+],[\mu^-])$
for $([\mu^+],[\mu^-]) \in T_B^+ \times T_B^-$ is called a {\it BMO embedding}.
Since $T_p \subset T_B$,
we have $\Lambda=\widetilde \Lambda|_{T_p^+ \times T_p^-}$. 

A simple arc $\Gamma$ in $\mathbb C$ that extends to $\infty$ in both directions 
is called a {\it chord-arc curve} 
if it is the image of $\mathbb R$ under a bi-Lipschitz homeomorphism of $\mathbb C$. 
Chord-arc curves are locally rectifiable. 
Let $\widetilde T_C$ be the open subset of $T_B^+ \times T_B^-$ consisting of BMO
embeddings $\gamma:\mathbb R \to \mathbb C$ whose images
are chord-arc curves. These embeddings are characterized by the property that
$\gamma$ is locally absolutely continuous and $|\gamma'|$ is an $A_\infty$-weight.
We note that $T_p^+ \times T_p^- \subset \widetilde T_C$.
It has been shown in \cite[Theorem 6.1]{WM-2} that $\widetilde \Lambda$ maps $\widetilde T_C$ biholomorphically onto its
image. 

\begin{proposition}\label{biholo-C}
$\widetilde \Lambda:\widetilde T_C \to {\rm BMO}(\mathbb R)$ is a biholomorphic homeomorphism onto its image.
In particular, the derivative $d\widetilde \Lambda$ is a Banach isomorphism onto ${\rm BMO}(\mathbb R)$
at every point of $\widetilde T_C$. 
\end{proposition}

Concerning item (a) in the above steps, we have the following.

\begin{proposition}\label{image}
For the holomorphic injection $\Lambda:T_p^+ \times T_p^- \to \widehat B_p(\mathbb R)$ with $p \geq 1$, 
the image ${\rm Ran}\,d_{([\mu^+],[\mu^-])}\, \Lambda$ of the derivative of $\Lambda$ at $([\mu^+],[\mu^-]) \in T_p^+ \times T_p^-$
is given as the algebraic direct sum
$$
{\rm Ran}\,d_{([\mu^+],[\mu^-])}\, \Lambda
=C_{h^-}(\widehat {\mathcal B}_p(\mathbb H^-)) \dot{+} C_{h^+}(\widehat {\mathcal B}_p(\mathbb H^+))
$$
of the two closed subspaces,
where $h^+=\pi(\mu^+)$ and $h^-=\pi(\mu^-)$ are the corresponding quasisymmetric homeomorphisms of $\mathbb R$.
\end{proposition}

\begin{proof}
Under the direct sum decomposition of the tangent space 
${\mathscr T}_{[\mu^+]}T_p^+ \oplus {\mathscr T}_{[\mu^-]}T_p^-$ as in \eqref{tangentspace},
we will show that
\begin{equation}\label{tangent}
d_{([\mu^+],[\mu^-])}\, \Lambda({\mathscr T}_{[\mu^+]}T_p^+)=C_{h^-}(\widehat {\mathcal B}_p(\mathbb H^-));\quad
d_{([\mu^+],[\mu^-])}\, \Lambda({\mathscr T}_{[\mu^-]}T_p^-)=C_{h^+}(\widehat {\mathcal B}_p(\mathbb H^+)).
\end{equation} 
To this end, we use relation \eqref{translation} to obtain
$$
\Lambda([\mu^+],[\mu^-])
=\Lambda \circ R_{[\mu^+]}([0],[\mu^-] \ast [\bar \mu^+]^{-1})
=Q_{h^+} \circ \Lambda([0],r_{[\bar \mu^+]}^{-1}([\mu^-])).
$$
By fixing $[\mu^+]$, we take the partial derivative of this formula along the direction of $T_p^-$. Then,
\begin{equation}
d_{([\mu^+],[\mu^-])}\, \Lambda|_{{\mathscr T}_{[\mu^-]}T_p^-}
=C_{h^+} \circ d_{([0],[\mu^-]\ast[\bar \mu^+]^{-1})}\Lambda \circ d_{[\mu^-]}r^{-1}_{[\bar \mu^+]},
\end{equation}
where $d_{([0],[\mu^-]\ast[\bar \mu^+]^{-1})}\Lambda$ can be regarded as the identity map on the tangent subspace
${\mathscr T}_{[\mu^-]\ast [\bar \mu^+]^{-1}}T_p^- \cong \widehat {\mathcal B}_p(\mathbb H^+)$.
This implies that $d_{([\mu^+],[\mu^-])}\, \Lambda({\mathscr T}_{[\mu^-]}T_p^-)=C_{h^+}(\widehat {\mathcal B}_p(\mathbb H^+))$.
The other equation in \eqref{tangent} is similarly proved.

Finally, we show that $C_{h^-}(\widehat {\mathcal B}_p(\mathbb H^-)) \cap C_{h^+}(\widehat {\mathcal B}_p(\mathbb H^+))=\{0\}$.
Suppose that $\phi \in \widehat B_p(\mathbb R)$ belongs to both
$C_{h^-}(\widehat {\mathcal B}_p(\mathbb H^-))$ and $C_{h^+}(\widehat {\mathcal B}_p(\mathbb H^+))$.
The derivative $d_{([\mu^+],[\mu^-])}\, \Lambda$ is injective because 
$$
d_{([\mu^+],[\mu^-])}\, \Lambda=d_{([\mu^+],[\mu^-])}\, \widetilde \Lambda|_{{\mathscr T}_{([\mu^+],[\mu^-])}(T_p^+ \times T_p^-)}
$$ 
and $d_{([\mu^+],[\mu^-])}\, \widetilde \Lambda$ is injective by Proposition \ref{biholo-C}
for $\widetilde \Lambda:\widetilde T_C \to {\rm BMO}(\mathbb R)$. Hence,
there is a unique tangent vector $u \in {\mathscr T}_{([\mu^+],[\mu^-])}(T_p^+ \times T_p^-)$ such that
$d_{([\mu^+],[\mu^-])}\, \Lambda(u)=\phi$. However, by belonging of $\phi$, $u$ has to lie in both ${\mathscr T}_{[\mu^-]}T_p^-$ 
and ${\mathscr T}_{[\mu^+]}T_p^+$, and thus $u=o$. This implies that $\phi=0$.
\end{proof}

\begin{remark}\label{direct}
After proving Theorem \ref{main2}, we see that the above algebraic direct sum in fact coincides with
$\widehat B_p(\mathbb R)$ and thus obtain the topological direct sum decomposition
\begin{equation}\label{sum}
\widehat B_p(\mathbb R)=C_{h^-}(\widehat {\mathcal B}_p(\mathbb H^-)) \oplus C_{h^+}(\widehat {\mathcal B}_p(\mathbb H^+)).
\end{equation}
\end{remark}

Item (b) is proved by using the following lemma. The assumption of this statement will be
verified by applying the corresponding result for $\widetilde \Lambda:\widetilde T_C \to {\rm BMO}(\mathbb R)$.

\begin{lemma}\label{imaginary}
If the real subspace $i{\rm Re}\,\widehat B_p(\mathbb R)$ is contained in
the image ${\rm Ran}\,d_{([\mu^+],[\mu^-])}\, \Lambda$ of the derivative
at $([\mu^+],[\mu^-]) \in T_p^+ \times T_p^-$, then $d_{([\mu^+],[\mu^-])}\, \Lambda$ is surjective, that is to say,
${\rm Ran}\,d_{([\mu^+],[\mu^-])}\, \Lambda=\widehat B_p(\mathbb R)$. 
\end{lemma}

\begin{proof}
It suffices to show that ${\rm Re}\,\widehat B_p(\mathbb R) \subset {\rm Ran}\,d_{([\mu^+],[\mu^-])}\, \Lambda$.
We note that
Proposition \ref{image} implies that 
$C_{h^+}(\widehat {\mathcal B}_p(\mathbb H^+)) \subset {\rm Ran}\,d_{([\mu^+],[\mu^-])}\, \Lambda$.

We take any $\phi \in {\rm Re}\,\widehat B_p(\mathbb R)$. Since $C_{h^+}$ maps ${\rm Re}\,\widehat B_p(\mathbb R)$
isomorphically onto itself, $C_{h^+}^{-1}(\phi)$ also belongs to ${\rm Re}\,\widehat B_p(\mathbb R)$.
We consider its Szeg\"o projection
\begin{equation}\label{PC}
P^+(C_{h^+}^{-1}(\phi))=\frac{1}{2}C_{h^+}^{-1}(\phi)+i \frac{1}{2}\mathcal H \circ C_{h^+}^{-1}(\phi),
\end{equation}
which is in $\widehat {\mathcal B}_p(\mathbb H^+)$. 
The application of $2C_{h^+}$ to this implies
$$
\phi +i C_{h^+} \circ \mathcal H \circ C_{h^+}^{-1}(\phi) \in C_{h^+}(\widehat {\mathcal B}_p(\mathbb H^+)) \subset
{\rm Ran}\,d_{([\mu^+],[\mu^-])}\, \Lambda.
$$
Here, 
$i C_{h^+} \circ \mathcal H \circ C_{h^+}^{-1}(\phi) \in i{\rm Re}\,\widehat B_p(\mathbb R)$ also belongs to
${\rm Ran}\,d_{([\mu^+],[\mu^-])}\, \Lambda$ by the assumption. Therefore, we have 
$\phi \in {\rm Ran}\,d_{([\mu^+],[\mu^-])}\, \Lambda$.
\end{proof}

\begin{remark}
Likewise to the above lemma, we can prove that
if ${\rm Re}\,\widehat B_p(\mathbb R)$ is contained in ${\rm Ran}\,d_{([\mu^+],[\mu^-])}\, \Lambda$, then
$d_{([\mu^+],[\mu^-])}\, \Lambda$ is surjective. 
\end{remark}

Having these preparatory arguments, we can provide the required proof as follows.
Conducting item (c) is included in it.

\begin{proof}[Proof of Theorem \ref{main2}]
Proposition \ref{holomorphic} asserts that $\Lambda$ is a holomorphic injection. 
To show that $\Lambda$ is biholomorphic, we prove that the derivative $d_{([\mu^+],[\mu^-])}\, \Lambda$ 
at every $([\mu^+],[\mu^-]) \in T_p^+ \times T_p^-$ is
surjective onto $\widehat{B}_p(\mathbb R)$. Then, by the inverse mapping theorem (see \cite[\S 7.18]{Ch}),
we will obtain the required claim.

Let $\phi=\Lambda([\mu^+],[\mu^-]) \in \widehat B_p(\mathbb R)$. We can find 
$\phi_0 \in i{\rm Re}\,\widehat{B}_p(\mathbb R) \cap {\rm Ran}\,\Lambda$ and $[\nu] \in T_p$
such that $Q_{h}(\phi_0)=\phi$ with $\pi(\nu)=h$. Indeed, we take $[\nu] \in T_p$ such that 
$\Lambda([\nu],[\bar \nu]) =\log h'={\rm Re}\,\phi$
by Lemma \ref{converse}, and set $\phi_0=iC_h^{-1}({\rm Im}\,\phi)$. Then,
$$
Q_{h}(\phi_0)=C_h(\phi_0)+\log h'=i{\rm Im}\,\phi+{\rm Re}\,\phi=\phi.
$$
Since $Q_{h}$ preserves ${\rm Ran}\,\Lambda$ by relation \eqref{translation}, we see that $\phi_0 \in {\rm Ran}\,\Lambda$.

We consider the derivative of $\Lambda$ at $([\mu^+_0],[\mu^-_0])=R_{[\nu]}^{-1}([\mu^+],[\mu^-])$,
where $\Lambda([\mu^+_0],[\mu^-_0])=Q_{h}^{-1} \circ \Lambda([\mu^+],[\mu^-])=\phi_0$ by \eqref{translation}.
For an arbitrary tangent vector
$v \in i{\rm Re}\,\widehat B_p(\mathbb R)$ at $\phi_0$, we take a sufficiently short segment $c(t)=\phi_0+tv$ in 
$i{\rm Re}\,\widehat B_p(\mathbb R)$ defined for $t \in (-\varepsilon,\varepsilon)$. 
Then, under the biholomorphic homeomorphism 
$\widetilde \Lambda:\widetilde T_C \to {\rm BMO}(\mathbb R)$ onto its image which contains $c(t)$,
each point in the inverse image $(\widetilde \Lambda^{-1} \circ c)(t)$ corresponds to a chord-arc curve of
arc-length parametrization generated by an element of $i{\rm Re}\,\widehat B_p(\mathbb R)$.
In this case, we can apply \cite[Proposition 5.5]{WM-4} to show that each $(\widetilde \Lambda^{-1} \circ c)(t)$
is a $p$-Weil--Petersson embedding in
$T_p^+ \times T_p^-$. Hence, the tangent vector $u=\frac{d}{dt}(\widetilde \Lambda^{-1} \circ c)(t)|_{t=0}$
at $([\mu_0^+],[\mu_0^-])$ belongs to ${\mathscr T}_{([\mu_0^+],[\mu_0^-])}(T_p^+ \times T_p^-)$ and
satisfies $d_{([\mu_0^+],[\mu_0^-])} \widetilde \Lambda(u)=v$. Then, we use the fact that
$d \widetilde \Lambda|_{{\mathscr T}_{([\mu_0^+],[\mu_0^-])}(T_p^+ \times T_p^-)}=d\Lambda$ (see \cite[Claim 2]{WM-4})
to conclude that $v \in {\rm Ran}\,d_{([\mu_0^+],[\mu_0^-])}\, \Lambda$.
Thus, we have $i{\rm Re}\,\widehat B_p(\mathbb R) \subset {\rm Ran}\,d_{([\mu_0^+],[\mu_0^-])}\, \Lambda$.

Under this condition,
Lemma \ref{imaginary} yields that $d_{([\mu_0^+],[\mu_0^-])}\, \Lambda$ is surjective.
Hence, by the inverse mapping theorem, 
there exists some neighborhood $U_0$ of $\phi_0$ in $\widehat B_p(\mathbb R)$ such that $U_0 \subset {\rm Ran}\,\Lambda$, and
$\Lambda^{-1}$ is holomorphic on $U_0$.
Since 
$$
d_{([\mu^+],[\mu^-])}\,\Lambda=d_{\phi_0}Q_{h} \circ d_{([\mu_0^+],[\mu_0^-])}\,\Lambda 
\circ d_{([\mu^+],[\mu^-])}R_{[\nu]}^{-1},
$$
this is also surjective (and hence isomorphic). Alternatively, we can construct the local holomorphic inverse
$R_{[\nu]} \circ \Lambda^{-1}|_{U_0} \circ Q_{h}^{-1}$ on the neighborhood $Q_{h}(U_0)$ of $\phi$.
\end{proof}

Finally, as an application of Theorem \ref{main2}, we obtain a result about
the real-analytic structure of the integrable Teich\-m\"ul\-ler space $T_p$ for $p \geq 1$.
We restrict the biholomorphic map $\Lambda$ to the real-analytic submanifold ${\rm Sym}\,(T_p^+ \times T_p^-)$, 
as in the setting of Theorem \ref{theorem10}. 
By composing $\Lambda$ with the canonical real-analytical embedding 
$\iota:T_p \to {\rm Sym}\,(T_p^+ \times T_p^-)$ given by $\iota([\mu])=([\mu],[\bar \mu])$, we have the following result. 
With some restrictions on $p$, it has appeared in \cite[Theorem 2.3]{ST} and \cite[Corollary 5.2]{WM-4}.

\begin{corollary}\label{structure}
For $p \geq 1$, the map $\Lambda \circ \iota:T_p \to {\rm Re}\,\widehat B_p(\mathbb R)$
given by $h \mapsto \log h'$
is a real-analytic diffeomorphism.
\end{corollary}

Thus, the integrable Teich\-m\"ul\-ler space $T_p$ is real-analytically equivalent to 
the entire real Banach space ${\rm Re}\,\widehat B_p(\mathbb R)$. Hence, 
the complex Banach structure of $T_p$ is
subordinate to the real Banach structure of ${\rm Re}\,\widehat B_p(\mathbb R)$.

\section{Analysis on curves via biholomorphic correspondence}\label{10}

The biholomorphic correspondence from Weil--Petersson embeddings in the Bers coordinates to the Besov space not only describes the structure of the integrable Teich\-m\"ul\-ler spaces but also provides a method of analysis on Weil--Petersson curves for classical real analysis problems. In this section,
we first introduce the Cauchy transform and the Cauchy projection by
considering the Cauchy integral on a Weil--Petersson curve $\Gamma = \gamma(\mathbb R)$,
where 
$\gamma = \gamma([\mu^+],[\mu^-])$ is a $p$-Weil--Petersson embedding for $([\mu^+],[\mu^-])
\in T_p^+ \times T_p^-$ with $p \geq 1$. 

We define the Banach space of Besov functions on the $p$-Weil--Petersson
curve $\Gamma=\gamma(\mathbb R)$ by
the push-forward of $\widehat B_p(\mathbb R)$ by $\gamma$ and identify this pair. Namely, 
$$
\widehat B_p(\gamma(\mathbb R))=\{\gamma_*\phi=\phi \circ \gamma^{-1} \mid \phi \in \widehat B_p(\mathbb R)\}
$$
with norm $\Vert \gamma_*\phi \Vert_{\widehat B_p(\gamma)}=\Vert \phi \Vert_{\widehat B_p}$. 

\begin{remark}
Usually, the function spaces on the locally rectifiable curve $\Gamma$ is defined
by using its arc-length parametrization.
In our case, for a $p$-Weil--Petersson  embedding $\gamma_0:\mathbb R \to \mathbb C$ with $\gamma_0(\mathbb R)=\Gamma$
such that $\gamma_0$ is the arc-length parametrization of $\Gamma$,
we may consider $\widehat B_p(\gamma_0(\mathbb R))$ in the above notation. However, the difference between
$\gamma$ and $\gamma_0$ is given by the composition operator $C_{h}$
for a quasisymmetric homeomorphism $h \in T_p$
(as explained in the next section), and it can be controlled well. Hence, we adopt
the representation $\widehat B_p(\gamma(\mathbb R))$ because it helps the arguments to be more natural
when we consider dependence of operators on $\gamma$.
\end{remark}

Let $\Omega^+$ and $\Omega^-$ be the left and the right domains bounded by $\Gamma$, respectively.
Let $F^\pm:\mathbb H^\pm \to \Omega^\pm$ be the normalized Riemann mappings.
Then, we also define the Banach space of analytic Besov functions on $\Omega^\pm$ 
by the push-forward of $\widehat {\mathcal B}_p(\mathbb H^\pm)$ by $F^\pm$ and 
identify these pairs. Namely,
$$
\widehat {\mathcal B}_p(\Omega^\pm)=\{(F^\pm)_* \Phi^\pm=\Phi^\pm \circ (F^\pm)^{-1} \mid \Phi^\pm \in \widehat {\mathcal B}_p(\mathbb H^\pm)\}
$$
with norm $\Vert (F^\pm)_* \Phi^\pm \Vert_{\widehat {\mathcal B}_p(\Omega^\pm)}=\Vert \Phi^\pm \Vert_{\widehat {\mathcal B}_p}$.

For any functions $(F^\pm)_* \Phi^\pm \in \widehat {\mathcal B}_p(\Omega^\pm)$, 
their boundary extensions to $\Gamma$ is defined by
$$
E_\Gamma^\pm((F^\pm)_* \Phi^\pm)=E(\Phi^\pm) \circ (F^\pm)^{-1},
$$
where the Riemann mappings $F^\pm$ are assumed to extend to the homeomorphisms of $\mathbb R$ onto $\Gamma$.

\begin{proposition}\label{extnorm}
It holds that $E_\Gamma^\pm(\widehat {\mathcal B}_p(\Omega^\pm)) \subset \widehat B_p(\gamma(\mathbb R))$
for any $\gamma=\gamma([\mu^+],[\mu^-])$.  
Moreover, these trace operators $E_\Gamma^\pm$ are Banach isomorphisms onto their images, where their operator norms are
estimated in terms of the norms of the composition operators $C_{h^\pm}$
for $h^\pm = \pi(\mu^\pm)$.
\end{proposition}

\begin{proof}
The norm on $\widehat {\mathcal B}_p(\Omega^\pm)$ is induced from $\widehat {\mathcal B}_p(\mathbb H^\pm)$
by $F^\pm$ whereas the norm of $\widehat B_p(\gamma(\mathbb R))$ is induced from $\widehat B_p(\mathbb R)$ by $\gamma$.
Because $\gamma=F^\pm \circ h^\pm$, 
the differences between $\Vert E^\pm \Vert$ and $\Vert E_\Gamma^\pm \Vert$
are caused by the composition operators $C_{h^\pm}$.
\end{proof}

By this proposition, we see that $E_\Gamma^\pm(\widehat {\mathcal B}_p(\Omega^\pm))$ 
are closed subspaces of $\widehat B_p(\gamma(\mathbb R))$. Moreover, under the Banach isomorphism $E_\Gamma^\pm$,
we may identify $E_\Gamma^\pm(\widehat {\mathcal B}_p(\Omega^\pm))$ with $\widehat {\mathcal B}_p(\Omega^\pm)$. Hence,
we regard $\widehat {\mathcal B}_p(\Omega^\pm)$ as closed subspaces of
$\widehat B_p(\gamma(\mathbb R))$ without noticing $E_\Gamma^\pm$ hereafter.
However, the Banach isomorphisms $E_\Gamma^\pm$ are not uniform but they depend on $\gamma$
as shown in Proposition \ref{extnorm}.

\begin{definition} 
The {\it Cauchy transform} of  
$\psi \in \widehat B_p(\gamma(\mathbb R))$ on the $p$-Weil--Petersson curve $\Gamma = \gamma(\mathbb R)$ 
(oriented by $\mathbb R$ 
consistently with $\pm$) is defined by the singular integral
$$
({\mathcal H}_\Gamma \psi)(\xi)={\rm p.v.} \frac{1}{\pi} \int_{\Gamma} \frac{\psi(z)}{\xi-z}dz
\quad(\xi \in \Gamma),
$$
where $dz=\gamma'(t)dt$.
The Cauchy integrals of $\psi$ on $\Gamma$ are defined by
$$
({P}^\pm_\Gamma \psi)(\zeta)=\frac{1}{\pi} \int_{\Gamma} \frac{\psi(z)}{\zeta-z}dz
\quad (\zeta \in \Omega^\pm),
$$
which are holomorphic functions on $\Omega^\pm$.
\end{definition}

The point-wise (a.e.) convergence of the Cauchy transform and the Cauchy integrals for $\widehat B_p(\gamma(\mathbb R))$ 
are guaranteed by using
the regularized kernel. By the Privalov theorem (see \cite[p.431]{Go}), if the Cauchy transform $({\mathcal H}_\Gamma \psi)(\xi)$ exists
a.e. on $\Gamma$, then the Cauchy integrals $({P}^\pm_\Gamma \psi)(\zeta)$ have non-tangential limits a.e. on $\Gamma$,
and vice versa.
The boundary extensions $E_\Gamma^\pm({P}^\pm_\Gamma \psi)$ of
${P}^\pm_\Gamma \psi$ to $\Gamma$ 
are also denoted by the same symbol ${P}^\pm_\Gamma \psi$ and called the {\it Cauchy projections} of $\psi$.
Later, we see that ${P}^\pm_\Gamma \psi \in \widehat {\mathcal B}_p(\Omega^\pm)$,
and the Cauchy transform ${\mathcal H}_\Gamma \psi$ is in $\widehat B_p(\gamma(\mathbb R))$. 

The {\it Plemelj formula} (also named after Sokhotski and Privalov) for the Riemann--Hilbert problem asserts
the following relation between the Cauchy transform and the Cauchy projections.
This is a generalization of the relation between the Hilbert transform and the Szeg\"o projections.

\begin{proposition}\label{Plemelj}
For a function $\psi \in \widehat B_p(\gamma(\mathbb R))$ 
on the $p$-Weil--Petersson curve $\Gamma=\gamma(\mathbb R)$, the Cauchy transform ${\mathcal H}_\Gamma$ and Cauchy projections $P^\pm_\Gamma$ satisfy
\begin{equation*}
P^+_\Gamma \psi = \frac{1}{2}(\psi + i{\mathcal H}_\Gamma \psi),\quad
P^-_\Gamma \psi = \frac{1}{2}(\psi - i{\mathcal H}_\Gamma \psi).
\end{equation*}
In other words,
\begin{equation*}
\psi = P^+_\Gamma \psi + P^-_\Gamma \psi,\quad
i{\mathcal H}_\Gamma \psi = P^+_\Gamma \psi - P^-_\Gamma \psi
\end{equation*}
holds.
\end{proposition}

For $([\mu^+],[\mu^-]) \in T_p^+ \times T_p^-$, we have obtained 
the images of the tangent spaces ${\mathscr T}_{[\mu^+]} T_p^+$ and
${\mathscr T}_{[\mu^-]} T_p^-$ under the derivative $d_{([\mu^+],[\mu^-])}\Lambda$ as in \eqref{tangent},
and since $d_{([\mu^+],[\mu^-])}\Lambda$ is surjective, which is the main ingredient of the proof of Theorem \ref{main2},
we have the topological direct sum decomposition \eqref{sum} of $\widehat B_p(\mathbb R)$
as mentioned in Remark \ref{direct}. Then, every $\phi \in \widehat B_p(\mathbb R)$ is uniquely
represented by $\phi=\phi^++\phi^-$ for $\phi^+ \in C_{h^+}(\widehat {\mathcal B}_p(\mathbb H^+))$ and
$\phi^- \in C_{h^-}(\widehat {\mathcal B}_p(\mathbb H^-))$, and this defines the bounded linear projections
\begin{align}
P^+_{([\mu^+],[\mu^-])}&:\widehat B_p(\mathbb R) \to C_{h^+}(\widehat {\mathcal B}_p(\mathbb H^+)),\quad
\phi \mapsto \phi^+;\\
P^-_{([\mu^+],[\mu^-])}&:\widehat B_p(\mathbb R) \to C_{h^-}(\widehat {\mathcal B}_p(\mathbb H^-)),\quad
\phi \mapsto \phi^-.
\end{align}

\begin{theorem}\label{Cauchy}
In the Besov space $\widehat B_p(\gamma(\mathbb R))$ on the $p$-Weil--Petersson curve $\Gamma=\gamma(\mathbb R)$ 
of $\gamma=\gamma([\mu^+],[\mu^-])$ for $p \geq 1$, 
the Cauchy projections $P^\pm_\Gamma$ satisfy
$$
P^\pm_\Gamma = \gamma_* \circ P^\pm_{([\mu^+],[\mu^-])} \circ \gamma_*^{-1}.
$$
In particular, $P^\pm_\Gamma$ maps $\widehat B_p(\gamma(\mathbb R))$ onto $\widehat {\mathcal B}_p(\Omega^\pm)$, they are bounded linear operators, 
and their operator norms are estimated in terms of $\Vert C_{h^\pm}\Vert$.
\end{theorem}

\begin{proof}
Let $\phi=\gamma_*^{-1}(\psi)=\psi \circ \gamma$ for $\psi \in \widehat B_p(\gamma(\mathbb R))=\gamma_*(\widehat B_p(\mathbb R))$. Then,
$\phi \in \widehat B_p(\mathbb R)$ and $ P^\pm_{([\mu^+],[\mu^-])} \circ \gamma_*^{-1}(\psi)=\phi^\pm
\in C_{h^\pm}(\widehat {\mathcal B}_p(\mathbb H^\pm))$. Therefore,
$$
\gamma_* \circ P^\pm_{([\mu^+],[\mu^-])} \circ \gamma_*^{-1}(\psi) \in
\widehat {\mathcal B}_p(\Omega^\pm)
$$
since $\gamma=F^\pm \circ h^\pm$ and $\widehat {\mathcal B}_p(\Omega^\pm)=F^\pm_* (\widehat {\mathcal B}_p(\mathbb H^\pm))$
for the Riemann mappings $F^\pm:\mathbb H^\pm \to \Omega^\pm$.
By the definition of the projections $P^\pm_{([\mu^+],[\mu^-])}$, we have
\begin{equation}\label{decomp1}
\gamma_* \circ P^+_{([\mu^+],[\mu^-])} \circ \gamma_*^{-1}(\psi)+\gamma_* \circ P^-_{([\mu^+],[\mu^-])} \circ \gamma_*^{-1}(\psi)=\psi.
\end{equation}
Here, if we consider a measurable function $\Psi_1$ on $\mathbb C$ defined by holomorphic functions 
$\gamma_* \circ P^+_{([\mu^+],[\mu^-])} \circ \gamma_*^{-1}(\psi)$ on $\Omega^+$ and
$\gamma_* \circ P^-_{([\mu^+],[\mu^-])} \circ \gamma_*^{-1}(\psi)$ on $\Omega^-$,
it is locally integrable and is of growth order $\Psi_1(z)=o(|z|)$ as $z \to \infty$.

The Cauchy projections 
$P^\pm_\Gamma(\psi)$ satisfy the same properties as $\gamma_* \circ P^\pm_{([\mu^+],[\mu^-])} \circ \gamma_*^{-1}(\psi)$.
They are
holomorphic functions on $\Omega^\pm$ such that 
the boundary extension satisfies $P^+_\Gamma(\psi)+P^-_\Gamma(\psi)=\psi$ by Proposition \ref{Plemelj}. Moreover,
the measurable function $\Psi_2$ on $\mathbb C$ defined by $P^\pm_\Gamma(\psi)$ on $\Omega^\pm$
is locally integrable and is of growth order $\Psi_2(z)=o(|z|)$ as $z \to \infty$.
These follow from the arguments in \cite[Section 3]{Se1}.
Then, $\bar \partial$-derivative of $\Psi_1-\Psi_2$ is $0$ on $\mathbb C$
in the distribution sense, and hence it is holomorphic on $\mathbb C$. The growth order 
$\Psi_1(z)-\Psi_2(z)=o(|z|)$ as $z \to \infty$ forces it to be a constant function.
Thus, we have $P^\pm_\Gamma(\psi)=\gamma_* \circ P^\pm_{([\mu^+],[\mu^-])} \circ \gamma_*^{-1}(\psi)$.
\end{proof}

\begin{remark}
If we assume the fact that $P^\pm_\Gamma(\psi)$ belongs to ${\mathcal B}_p(\Omega^\pm)$
for $\psi \in B_p(\gamma(\mathbb R))$ taking a larger $p>1$ (see \cite[Corollary 5.4]{LS2}), 
then the proof of Theorem \ref{Cauchy} becomes easier.
Indeed, we have only to apply the direct sum decomposition \eqref{sum} to
$$
P^+_{([\mu^+],[\mu^-])}(\phi)-\gamma_*^{-1} \circ P^+_\Gamma \circ \gamma_*(\phi)=
\gamma_*^{-1} \circ P^-_\Gamma \circ \gamma_*(\phi)
-P^-_{([\mu^+],[\mu^-])}(\phi)
$$
for $\phi=\psi \circ \gamma \in B_p(\mathbb R)$.
\end{remark}

\begin{remark}
$P^\pm_\Gamma \circ \gamma_*=\gamma_* \circ P^\pm_{([\mu^+],[\mu^-])}$ maps $\widehat B_p(\mathbb R)$ onto 
$\widehat {\mathcal B}_p(\Omega^\pm)$ and ${\rm BMO}(\mathbb R)$ onto ${\rm BMOA}(\Omega^\pm)$.
When $[\mu^+]=[0]$ or $[\mu^-]=[0]$, this coincides with what is called the {\it Faber operator}
though the curve $\Gamma$ has assumed certain smoothness according to the function spaces in our case.
See \cite{LS1, LS2, WWH} for related arguments to ours. In our setting, we see not only the boundedness of the operator but
also its holomorphic dependence when the embeddings vary in the Teich\-m\"ul\-ler space as is discussed in the next section.
\end{remark}

This leads to {\it the Calder\'on theorem} (see \cite{CM0}) for Weil--Petersson curves. 
For chord-arc curves, the argument on the Bers coordinates for BMO embeddings in \cite{WM-2} can be 
also developed to the methods described here, which will appear in \cite{Ma}.
The boundedness of the Cauchy transform and projections in ${\rm BMO}(\gamma(\mathbb R))$ is verified in \cite{LS1}.
Results in a more general setting of
the Cauchy transform and projections in $B_p(\gamma(\mathbb R))$ for $p>1$ can be found in \cite{LS2}.

\begin{corollary}\label{Calderon}
The Cauchy projections $P_\Gamma^\pm$ on the $p$-Weil--Petersson curve $\Gamma=\gamma(\mathbb R)$ 
for $p \geq 1$ are associated with the topological direct sum decomposition
$$
\widehat B_p(\gamma(\mathbb R)) = \widehat {\mathcal B}_p(\Omega^+) \oplus \widehat {\mathcal B}_p(\Omega^-).
$$
\end{corollary}

\begin{remark}
Both the conjugate 
$C_{h^\pm} \circ P^\pm \circ C_{h^\pm}^{-1}$ of the Szeg\"o projections
and the {\it standardized} Cauchy projections $P^{\pm}_{([\mu^+],[\mu^-])}$ map ${\widehat B}_p(\mathbb R)$
onto $C_{h^\pm}({\mathcal B}_p(\mathbb H^\pm)) \subset {\widehat B}_p(\mathbb R)$, 
but they are different. In the special case $p=2$, ${B}_2(\mathbb R)$ is a Hilbert space and
we can think of the conjugate operator for any bounded linear operator. Then, the Szeg\"o projection is an orthogonal projection
and self-adjoint, but the Cauchy projection is not in general. They are related as follows:
\begin{equation}
(C_{h^\pm} \circ P^\pm \circ C_{h^\pm}^{-1})\circ (I+P_{([\mu^+],[\mu^-])}^\pm- P^{\pm \qquad \ast}_{([\mu^+],[\mu^-])})
=P_{([\mu^+],[\mu^-])}^\pm.
\end{equation}
This is the same as the {\it Kerzman--Stein formula} (see \cite{KS}) for $L_2(\mathbb R)$ and the Hardy space
${H}_2(\mathbb H^\pm)$.
\end{remark}

\section{Holomorphic dependence of the Cauchy transform}\label{11}

In this section, we consider how the Cauchy transforms ${\mathcal H}_\Gamma$ vary
when $\Gamma=\gamma(\mathbb R)$ move according to $\gamma=\gamma([\mu^+],[\mu^-])$.
To formulate this problem, we take the conjugate of ${\mathcal H}_\Gamma$ so that 
it acts on ${\widehat B}_p(\mathbb R)$.

Let $\gamma=\gamma([\mu^+],[\mu^-])$ be a $p$-Weil--Petersson embedding
for $([\mu^+],[\mu^-]) \in T_p^+ \times T_p^-$ with $p \geq 1$.
By Proposition \ref{Plemelj}, the relationship between the Cauchy transform ${\mathcal H}_\Gamma$
on the $p$-Weil--Petersson curve $\Gamma=\gamma(\mathbb R)$ and the Cauchy projections $P_\Gamma^\pm$ is given as
\begin{equation}\label{transform}
{\mathcal H}_\Gamma=-i(P_\Gamma^+-P_\Gamma^-).
\end{equation}
Moreover, Theorem \ref{Cauchy} shows $P^\pm_\Gamma=\gamma_* \circ P^\pm_{([\mu^+],[\mu^-])} \circ \gamma_*^{-1}$.
Then, pulling back ${\mathcal H}_\Gamma$ to $\mathbb R$ via $\gamma$, 
we define
\begin{equation}\label{HP}
{\mathcal H}_{([\mu^+],[\mu^-])}=-i( P^+_{([\mu^+],[\mu^-])} -P^-_{([\mu^+],[\mu^-])})
=\gamma_*^{-1} \circ {\mathcal H}_\Gamma \circ \gamma_*,
\end{equation}
which is a Banach automorphism of $\widehat B_p(\mathbb R)$.
We refer to ${\mathcal H}_{([\mu^+],[\mu^-])}$
as the standardization of the Cauchy transform ${\mathcal H}_\Gamma$.
More explicitly,
$$
{\mathcal H}_{([\mu^+],[\mu^-])}(\phi)(x)=
{\rm p.v.} \frac{1}{\pi} \int_{\Gamma} \frac{\phi \circ \gamma^{-1} (z)}{\gamma(x)-z}dz
=
{\rm p.v.} \frac{1}{\pi} \int_{-\infty}^\infty \frac{\phi(t)}{\gamma(x)-\gamma(t)}\gamma'(t)dt
\quad(x \in \mathbb R)
$$
for $\phi \in {\widehat B}_p(\mathbb R)$.

For $([0],[0]) \in T_p^+ \times T_p^-$, 
${\mathcal H}_{([0],[0])}$ coincides with the Hilbert transform $\mathcal H$. 
For $([\mu],[\bar \mu]) \in {\rm Sym}\,(T_p^+ \times T_p^-)$, 
${\mathcal H}_{([\mu],[\bar \mu])}$
is the conjugate of $\mathcal H$ by the composition operator $C_h$ for $h=\pi(\mu) \in T_p$. 
Indeed, \eqref{HP} applied to the case of $\Gamma=h(\mathbb R)=\mathbb R$ with 
$\gamma([\mu],[\bar \mu])=H(\mu)|_{\mathbb R}=h$ yields ${\mathcal H}_{([\mu],[\bar \mu])}=C_h \circ {\mathcal H} \circ C_h^{-1}$.

Let ${\mathcal L}(\widehat B_p(\mathbb R))$ be
the Banach space of all bounded linear operators $\widehat B_p(\mathbb R) \to \widehat B_p(\mathbb R)$
equipped with the operator norm. We consider the map
$$
\eta:T_p^+ \times T_p^- \to {\mathcal L}(\widehat B_p(\mathbb R))
$$
defined by $([\mu^+],[\mu^-]) \mapsto {\mathcal H}_{([\mu^+],[\mu^-])}$.

\begin{theorem}\label{conjugateH}
$\eta:T_p^+ \times T_p^- \to {\mathcal L}(\widehat B_p(\mathbb R))$ is holomorphic for $p \geq 1$.
\end{theorem}

\begin{proof}
Due to formula \eqref{HP}, it suffices to show that $P^\pm_{([\mu^+],[\mu^-])}$ 
depend holomorphically on $([\mu^+],[\mu^-]) \in T_p^+ \times T_p^-$.
We consider the derivative $d_{([\mu^+],[\mu^-])}\,\Lambda$ of the biholomorphic map 
$\Lambda$ and its inverse 
$$
(d_{([\mu^+],[\mu^-])}\,\Lambda)^{-1}:
\widehat B_p(\mathbb R) \to {\mathscr T}_{([\mu^+],[\mu^-])}(T_p^+ \times T_p^-)=
{\mathscr T}_{[\mu^-]}(T_p^-) \oplus {\mathscr T}_{[\mu^+]}(T_p^+).
$$
According to this direct sum decomposition of the tangent space, we denote the canonical projections by
$$
J^+:{\mathscr T}_{([\mu^+],[\mu^-])}(T_p^+ \times T_p^-) \to {\mathscr T}_{[\mu^-]}(T_p^-),\quad
J^-:{\mathscr T}_{([\mu^+],[\mu^-])}(T_p^+ \times T_p^-) \to {\mathscr T}_{[\mu^+]}(T_p^+).
$$
Then, by the proof of formula \eqref{tangent} for Proposition \ref{image}, we see that
the standardizations of the Cauchy projections turn out to be
\begin{align}
P^\pm_{([\mu^+],[\mu^-])}&=d_{([\mu^+],[\mu^-])}\,\Lambda\circ J^\pm\circ (d_{([\mu^+],[\mu^-])}\,\Lambda)^{-1}
=d_{([\mu^+],[\mu^-])}\,\Lambda\circ J^\pm\circ d_{\Lambda([\mu^+],[\mu^-])}\,\Lambda^{-1}.
\end{align}
Since $\Lambda$ is biholomorphic, these operators depend holomorphically on $([\mu^+],[\mu^-]) \in T_p^+ \times T_p^-$.
\end{proof}

Corresponding results for BMO embeddings with chord-arc image are found in \cite{CM0}, 
though without involving the biholomorphic map $\Lambda$, only demonstrating holomorphic dependence near the origin. 
A complete extension is also possible, which is in \cite{Ma}.

As an application of the holomorphic dependence of ${\mathcal H}_{([\mu^+],[\mu^-])}$,
we consider the modification of the
Coifman--Meyer theorem which was originally proved for chord-arc curves.

We define $Z_p = i\widehat B_p(\mathbb R) \cap {\rm Ran}\,\Lambda$ as the real-analytic submanifold of ${\rm Ran}\,\Lambda$ consisting of purely imaginary-valued Besov functions (which is an open subset of the real Banach subspace
$i\widehat B_p(\mathbb R)$) as used in the proof of Theorem \ref{main2}. For $\psi \in Z_p$, 
$$
\gamma_0(x) = \int_0^x \exp\psi(t)dt \quad (x \in \mathbb R)
$$ 
serves as the {\it arc-length parameterization} of the $p$-Weil--Petersson curve $\gamma_0(\mathbb R)$.
In general, for any $p$-Weil--Petersson embedding 
$$
\gamma(x) = \int_0^x \exp \phi(t)dt
$$ 
determined by $\phi \in {\rm Ran}\,\Lambda$,
we take the quasisymmetric homeomorphism 
$$
h(x) = \int_0^x \exp({\rm Re}\,\phi(t))dt
$$ 
and set $\psi = i\, {\rm Im}\,\phi \circ h^{-1} \in Z_p$. 
Then, 
the arc-length parameter $\gamma_0$ determined by $\psi$ allows $\gamma$ to be expressed as the {\it reparameterization} of $\gamma_0$ by $h$, that is,
$\gamma = \gamma_0 \circ h$. See \cite[Lemma 4.2]{WM-4}.

We define $Y_p = \widehat {\mathcal B}_p(\mathbb H^+) \cap {\rm Ran}\,\Lambda$ as the complex submanifold of ${\rm Ran}\,\Lambda$ consisting of analytic Besov functions on $\mathbb H^+$. For $\varphi \in Y_p$, the corresponding $p$-Weil--Petersson embedding
$$
f(x) = \int_0^x \exp \varphi(t)dt \quad (x \in \mathbb R),
$$ 
when applied as above, can be expressed as the reparameterization of the arc-length parameter $\gamma_0$ by a quasisymmetric map $h$
so that
$f = \gamma_0 \circ h$.
This correspondence between $f$ and $\gamma_0$ is bijective, thus defining a mapping from $Z_p$ to $Y_p$,
which will be recognized later as a homeomorphism. 
Similarly, the correspondence from the pair $(f, \gamma_0)$ to their reparameterization $h$ allows 
the definition of a mapping from $Z_p$ to ${\rm Re}\,\widehat{B}_p(\mathbb R)$.

To investigate these mappings on $T_p^+ \times T_p^-$ through the biholomorphic homeomorphism $\Lambda$, we define
\begin{align*}
\rho&: T_p^+ \times T_p^- \to \{[0]\} \times T_p^-, \quad ([\mu^+],[\mu^-]) \mapsto ([0], [\mu^-] \ast [\overline{\mu^+}]^{-1});\\
\delta&: T_p^+ \times T_p^- \to {\rm Sym}\,(T_p^+ \times T_p^-), \quad ([\mu^+],[\mu^-]) \mapsto ([\mu^+],[\overline{\mu^+}]).
\end{align*}
Here, $\rho$ is continuous due to the topological group structure of $T_p$
given in Proposition \ref{group}, and $\delta$ is the projection to the symmetric axis, which is real-analytic. 
The unique decomposition 
$$
([\mu^+],[\mu^-]) = \rho([\mu^+],[\mu^-]) \ast \delta([\mu^+],[\mu^-])
$$
corresponds to the decomposition of the $p$-Weil--Petersson embedding $g=\gamma([\mu^+],[\mu^-])$ into $g=f \circ h$, where
$f=\gamma(\rho([\mu^+],[\mu^-]))$ is the boundary extension of the conformal homeomorphism of $\mathbb H^+$ to $\mathbb R$, and
$h=\gamma(\delta([\mu^+],[\mu^-]))$ is a quasisymmetric homeomorphism of $\mathbb R$.

We transform the two maps defined on the submanifolds in $\widehat B_p(\mathbb R)$ into those in $T_p^+ \times T_p^-$
by $\Lambda$.
The map $Z_p \to Y_p$ corresponds to
$$
\rho_0 = \rho|_{\Lambda^{-1}(Z_p)}: \Lambda^{-1}(Z_p) \to \{[0]\} \times T_p^-=\Lambda^{-1}(Y_p),
$$
and the map $Z_p \to {\rm Re}\,\widehat{B}_p(\mathbb R)$ corresponds to
$$
\delta_0 = \delta|_{\Lambda^{-1}(Z_p)}: 
\Lambda^{-1}(Z_p) \to {\rm Sym}\,(T_p^+ \times T_p^-)=\Lambda^{-1}({\rm Re}\,{\widehat B}_p(\mathbb R)).
$$

For the map $Z_p \to Y_p$ and $\rho_0$, we know the following.
This has been proved in \cite[Proposition 6.5]{WM-4} for $p>1$, and the proof is essentially the same in this case.

\begin{proposition}
$\rho_0:\Lambda^{-1}(Z_p) \to \{[0]\} \times T_p^-=\Lambda^{-1}(Y_p)$ is a homeomorphism
for $p \geq 1$.
\end{proposition}

\begin{proof}
Since the map $Z_p \to Y_p$ is bijective, so is $\rho_0$.
Since $\rho$ is continuous, so is $\rho_0$.
Hence, it suffices to show that $\rho_0^{-1}$ is continuous.
To this end, we consider the conjugation of $\rho_0^{-1}$ by $\Lambda$, that is,
$\Lambda \circ \rho^{-1}_0 \circ \Lambda^{-1}:Y_p \to Z_p$.
For any $\varphi \in Y_p$, we have
$$
\Lambda \circ \rho^{-1}_0 \circ \Lambda^{-1}(\varphi)=C_h^{-1}(i{\rm Im}\,\varphi)
$$
with $h(x)=\int_0^x \exp({\rm Re}\,\varphi(t))dt$. Since $h \mapsto h^{-1}$ is continuous in $T_p$
by Proposition \ref{group}, we see that
$\widehat{B}_p(\mathbb R) \times T_p \to \widehat{B}_p(\mathbb R)$ defined by
$(\phi,h) \mapsto C_h^{-1}(\phi)$ is continuous by Lemma \ref{strong} below. 
Thus, $\Lambda \circ \rho^{-1}_0 \circ \Lambda^{-1}$ is continuous.
\end{proof}

\begin{lemma}\label{strong}
The map $\widehat{B}_p(\mathbb R) \times T_p \to \widehat{B}_p(\mathbb R)$ defined by
$(\phi,h) \mapsto C_h(\phi)$ is continuous. In particular, a sequence of bounded linear operators
$C_{h_n}$ on $\widehat{B}_p(\mathbb R)$ converges to $C_h$ strongly as $h_n \to h$ in $T_p$.
\end{lemma}

\begin{proof}
This can be seen in \cite[Theorem 6.3]{WM-4} and the remark
after that in the case $p>1$. The same argument can be applied to $p=1$.
\end{proof}

\begin{remark}
In the circumstances of the above lemma, 
we do not know whether or not $C_{h_n}$ converges to $C_h$ in the operator norm.
\end{remark}

Coifman and Meyer \cite{CM} investigated the map $\delta_0$ for 
the BMO embeddings with chord-arc image. We
translate their results to the case of Weil--Petersson embeddings by formalizing the following. 
The proof follows immediately from what has been demonstrated earlier. 
For $p>1$, this has been obtained in \cite[Theorem 5.4]{WM-4}.

\begin{theorem}
$\delta_0: \Lambda^{-1}(Z_p) \to {\rm Sym}\,(T_p^+ \times T_p^-) = \Lambda^{-1}({\rm Re}\,\widehat{B}_p(\mathbb R))$ is a real-analytic diffeomorphism onto its image for $p \geq 1$.
\end{theorem}

\begin{proof}
Since $\delta_0$ is a real-analytic injection as is shown in the case $p>1$, 
it suffices to show that its inverse $\delta_0^{-1}$ is also real-analytic. 
As before, we consider the conjugate
$$
\Lambda \circ \delta^{-1}_0 \circ \Lambda^{-1}:{\rm Re}\,\widehat{B}_p(\mathbb R) \cap {\rm Ran}\,(\Lambda \circ \delta_0) \to Z_p.
$$
Then, for any $\phi \in {\rm Re}\,\widehat{B}_p(\mathbb R) \cap {\rm Ran}\,(\Lambda \circ \delta_0)$, we have
\begin{equation}
\Lambda \circ \delta^{-1}_0 \circ \Lambda^{-1}(\phi)=-i{\mathcal H}_{\Lambda^{-1}(\phi)} \phi
\end{equation}
by formula (*) in the proof of \cite[Theorem 5.4]{WM-4}. 
Because the standardized Cauchy transform ${\mathcal H}_{\Lambda^{-1}(\phi)}$
depends real-analytically on $\phi \in {\rm Re}\,\widehat{B}_p(\mathbb R)$
by Theorem \ref{conjugateH}, we see that $\Lambda \circ \delta^{-1}_0 \circ \Lambda^{-1}$ is real-analytic.
\end{proof}

%
%

\end{document}